\documentclass[12pt, leqno]{article}

\usepackage{physics}
\usepackage{amsmath}
\usepackage{tikz}
\usepackage{mathdots}
\usepackage{yhmath}
\usepackage{cancel}
\usepackage{color}
\usepackage{siunitx}
\usepackage{array}
\usepackage{multirow}
\usepackage{amssymb}
\usepackage{gensymb}
\usepackage{tabularx}
\usepackage{booktabs}
\usetikzlibrary{fadings}
\usetikzlibrary{patterns}
\usetikzlibrary{shadows.blur}
\usetikzlibrary{shapes}

\RequirePackage{latexsym} \RequirePackage{amsthm}
\RequirePackage{amsmath,xcolor}
\usepackage{enumerate}
\RequirePackage{amssymb} \RequirePackage{makeidx}
\usepackage{dsfont}
\usepackage{mathrsfs}
\numberwithin{equation}{section}
\usepackage[cp1251]{inputenc}

\newtheorem{conjecture}{\sc  Conjecture\rm}[section]
\newtheorem{preremark}[conjecture]{Remark}
\newenvironment{remark}%
  {\begin{preremark}\upshape}{\end{preremark}}
\newtheorem{predefinition}[conjecture]{Definition}
\newenvironment{definition}%
  {\begin{predefinition}\upshape}{\end{predefinition}}
\newtheorem{prelemma}[conjecture]{Lemma}
\newenvironment{lemma}%
  {\begin{prelemma}\upshape}{\end{prelemma}}
\newtheorem{preproposition}{Proposition}
\newenvironment{proposition}%
 {\begin{preproposition}\upshape}{\end{preproposition}}
\newtheorem{precorollary}{Corollary}
\newenvironment{corollary}%
  {\begin{precorollary}\upshape}{\end{precorollary}}
\newtheorem{pretheorem}{Theorem}
\newenvironment{theorem}%
  {\begin{pretheorem}\upshape}{\end{pretheorem}}

\def\R{{\mathbb R}}
\def\N{{\mathbb N}}

\def\tu{\tilde u}

\def\be{\begin{equation}}
\def\ee{\end{equation}}
\def\bea#1\eea{\begin{align}#1\end{align}}
\def\non{\nonumber}
\begin{document}

\title{\bf Entire minimizers of  Allen-Cahn systems with sub-quadratic potentials}

\author{Nicholas D. Alikakos\thanks{Department of Mathematics, University of Athens (EKPA), Panepistemiopolis, 15784 Athens,
Greece \,(nalikako@math.uoa.gr )} , {Dimitrios Gazoulis }\thanks{Department of Mathematics and Applied Mathematics, University of Crete, 70013 Heraklion,
Greece\hspace{10ex} (dgazoulis@math.uoa.gr )} , Arghir Zarnescu\thanks{
IKERBASQUE, Basque Foundation for Science, Plaza Euskadi 5
48009 Bilbao, Bizkaia, Spain (azarnescu@bcamath.org) }\,\,\thanks{BCAM, Basque Center for Applied Mathematics, Mazarredo 14, E48009 Bilbao, Bizkaia, Spain}\,\,\thanks{``Simion Stoilow" Institute of the Romanian Academy, 21 Calea Grivitei, 010702 Bucharest, Romania.}}

\date{}

\maketitle

\begin{center}
\textit{Dedicated to Pavol Brunovsky, a man of brilliance and very high morality}
\end{center}

\begin{abstract}
We study entire  minimizers of the  Allen-Cahn systems. The specific feature of our systems are  potentials having a finite number of global minima, with sub-quadratic behaviour locally near their minima.  The corresponding formal Euler-Lagrange equations are supplemented with free boundaries.
\par We do not study regularity issues but focus on qualitative aspects. We show the existence of entire solutions in an  equivariant setting connecting the minima of $ W $ at infinity, thus modeling many coexisting phases, possessing free boundaries and minimizing energy in the symmetry class. We also present a very modest result of existence of free boundaries under no symmetry hypotheses. The existence of a free boundary can be related to the existence of a specific sub-quadratic feature, a dead core, whose size is also quantified.
\end{abstract}

\section{Introduction and Main Results}

In this note we consider  minimizers in the whole space  $\mathbb{R}^n$ for  the functional

\begin{equation}\label{DefJu}
J (u) = \int \dfrac{1}{2} |\nabla u|^2 + W(u)dx
\end{equation} with  $ u : {\mathbb{R}}^n \rightarrow {\mathbb{R}}^m.$


\bigskip
We take $ W \geq 0 $ and $ \lbrace W = 0 \rbrace = \lbrace a_1, ...,a_N \rbrace:=A $, for some distinct points $ a_1,...,a_N \in {\mathbb{R}}^m $ that can physically model the phases of a substance that can exist in $ N \geq 2 $ equally preferred states. 

We assume that
\begin{equation}\label{liminfW}
\liminf_{|z|\rightarrow \infty} W(z) >0
\end{equation}

If $W$ is smooth then the first derivatives vanish at the minimum points and the generic local behaviour near such a minimum, say $a_i$,  is locally of quadratic nature, of the type $|u-a_i|^2$. The minimizers satisfy the Euler-Lagrange system

\begin{equation}\label{ELequation}
\Delta u -W_u(u) = 0.
\end{equation}

\bigskip
We are interested in the class of solutions that connect in some way the phases or a subset of them. The scalar case $ m=1 $ has been extensively studied with $ N=2 $ that is the natural choice. The reader may consult \cite{dPKW}, \cite{Savin}, \cite{W} where further references can be found. A well known conjecture of De Giorgi (1978) and its solution about thirty years later, played a significant role in the development of a large part of this work.

The vector case $ m\geq 2 $ by comparison has been studied very little. We note that for coexistence of three or more phases a vector order parameter is necessary and so there is physical interest for the system.

For $ m\geq 2 $, \eqref{ELequation} has been mainly studied in the class of equivariant solutions with respect to reflection groups beginning with \cite{BGS} and later \cite{GS} and significantly extended and generalized in various ways \cite{AF3}, \cite{A2}, \cite{F2}, \cite{AF1}, \cite{AF4}, \cite{BFS}. We refer to \cite{AFS} where existence under symmetry is covered and where more references can be found.

Degenerate, super-quadratic behavior at the minima has also been considered for \eqref{ELequation}, $ m=1$, in \cite{BS}, \cite{DFV}.

The focus of our work will be on going beyond this classical setting and explore the phenomena that are associated having sub-quadratic behaviour at the minima. Specifically, our  potentials are modelled near their minima $ a \in A$ after $ | u-a |^{\alpha} $, for $0< \alpha<2$. Furthermore we will consider also the limiting case $ \alpha =0 $ (that appears in a $ \Gamma $-limit setting as $ \alpha \rightarrow 0 $). 
Formally, the minimizers solve certain free boundary problems:
\begin{enumerate}
\item For $\alpha\in (0,2)$:
\begin{equation}\label{formal(0,2)}
\begin{cases} \Delta u = W_u(u) \;\;\; \textrm{for} \;\: \lbrace u(x) \notin A \rbrace \\ | \nabla u |^2 = 0 \;\;\; \textrm{for} \;\: \partial \lbrace u(x) \notin A \rbrace\end{cases}
\end{equation}
\item For $\alpha=0$:
\begin{equation}\label{formal0}
\begin{cases} \Delta u = 0 \;\;\; \textrm{for} \;\: \lbrace u(x) \notin A \rbrace \\ | \nabla u |^2 = 2 \;\;\; \textrm{for} \;\: \partial \lbrace u(x) \notin A \rbrace
\end{cases}
\end{equation}
In Appendix \ref{sec:appendixB} we give a formal  justification of these, that can be made rigorous with suitable regularity results, \cite{AGZ}. We note that for $\alpha=2$, Corollary $3.1$ p.$92$ in \cite{AFS} states that if both $W(u(x))=0$ and $|\nabla u(x)|^2=O(W(u(x))$ then $u\equiv a_i$. This latter condition holds in the scalar case, $m=1$, by the Modica inequality. Hence for $\alpha=2, m=1$ we have $\partial\lbrace u(x)\not\in A\rbrace=\emptyset$. Thus a free boundary may be expected only in the non smooth case. The reason is rather simple and can be traced back to the non-uniqueness of the trivial solution of the ODE $ u' = \frac{2}{2-\alpha} C^{\frac{\alpha}{2}}u^{\frac{\alpha}{2}} $ that describes the behavior of the one-dimensional solutions (connections) near the minimum of $ W $ of \eqref{formal(0,2)}, \eqref{formal0}.
\end{enumerate}

Thus we focus on the range $ 0 \leq \alpha <2 $. An important special case of the potentials we consider is given, for the set of minima $A=\lbrace a_1,\dots,a_N\rbrace $ by 
\begin{equation}\label{Walpha}
W^{\overline{\alpha}}(u) = \prod_{k=1}^N | u-a_k|^{{\alpha}_k} \;\;\;,\; \overline{\alpha} = ({\alpha}_1,...,{\alpha}_N); 0<\alpha_k<2,\: \forall \: k\in\{1,\dots,N\} 
\end{equation} 

More generally, motivated by the form of $ W $ in \eqref{Walpha}, we assume:

\begin{align*}
\textbf{(H1)} 
\begin{cases} \underline{0 < \alpha <2}: 
W \in C( {\mathbb{R}}^m ; [0, + \infty)) \; \textrm{with }\; \lbrace W = 0 \rbrace = \lbrace a_1,...,a_N \rbrace \neq \emptyset \; (N \geq 2) . \\
\textrm{For} \;\:  a\in \: \lbrace W = 0 \rbrace \; \textrm{the function} \; W \; \textrm{is differentiable in a deleted} \\ 
\textrm{neighborhood of} \; a \; \textrm{and satisfies} \; \frac{d}{d \rho } W(a + \rho \xi ) \geq \alpha C^* \rho^{ \alpha -1 } \:,\: \forall \; \rho \in (0, \rho_0] \:,\\ \forall \; \xi \in {\mathbb{R}}^m : | \xi |=1,\; 
\textrm{for some constants} \; \rho_0 >0, C^* >0 \; \textrm{independent of} \; \alpha.
\end{cases}
\end{align*}
\begin{align*}
\begin{cases}
\underline{\alpha = 0}: \lbrace W = 0 \rbrace = \lbrace a_1, ...,a_N \rbrace := A \;,\: W(u) := W^0 (u) := { \chi }_{ \lbrace u \in S_A \rbrace} \\ S_A := \lbrace \sum_{i=1}^N {\lambda}_i a_i \;,\: {\lambda}_i \in [0,1) \;,\: \forall \;\: i=1,...,N \;, \sum_{i=1}^N {\lambda}_i =1 \;,\: N=m+1 \rbrace\\ \textrm{We assume that the simplex} \; S_A \; \textrm{is nondegenerate, that is the vectors} \\ \lbrace a_2-a_1,...,a_{m+1}-a_1 \rbrace \; \textrm{are linearly independent} \; \textrm{and} \; m \geq 2.
\end{cases}
\end{align*}
$ \\ $
Clearly $ W^{\overline{\alpha}} $ in \eqref{Walpha} satisfy \textbf{(H1)} ($ 0 <\alpha<2 $).

We are primarily interested in bounded minimizers defined on $ {\mathbb{R}}^n $. We note in passing that the only critical points of $ J_{{\mathbb{R}}^n} \;\:, n\geq 2 $, with bounded energy are trivial \cite{A1}. A minimizer $ u $, by definition minimizes energy subject to its Dirichlet values on any open, bounded $ \Omega \subset {\mathbb{R}}^n $. More precisely,$ \\ $

\begin{definition}\label{DefMin} Let $ \mathcal{O} \subset {\mathbb{R}}^n $ open. A map $ u \in W^{1,2}_{loc}(\mathcal{O}, {\mathbb{R}^m}) \cap L^{\infty} (\mathcal{O} ; {\mathbb{R}}^m) $ is called a \textit{minimizer} of the energy functional $ J $ defined in \eqref{DefJu} if
\begin{equation}\label{Defmin}
J_{\Omega} (u+v) \geq J_{\Omega}(u) \;\;,\; \textrm{for} \;\: v \in W^{1,2}_0(\Omega, {\mathbb{R}^m}) \cap L^{\infty} (\Omega ; {\mathbb{R}}^m)
\end{equation}
for every open bounded Lipschitz set $ \Omega \subset \mathcal{O} $, with $ J_{\Omega} $ denoting the value of the integral in \eqref{DefJu} when integrating over the domain $ \Omega. \\ $
\end{definition}

The case, $ \alpha =0 $ for $ m=1 $ was introduced and extensively studied by Caffarelli and his collaborators, with particular attention to the optimal regularity of the solution and to the regularity of the free boundary. These are important classical results that can be found for example in the books \cite{CS} or \cite{PSU}. There is recent interest in the vector case for free boundary problems. We mention below two papers which relate to our work and where additional references can be found.

In \cite{CSY} the authors study minimizers of the functional
\begin{equation}\label{DefFunctCSY}
\int_{\Omega} (\frac{1}{2} | \nabla u |^2 + Q^2(x) \chi_{\lbrace |u|>0 \rbrace})dx
\end{equation}
with $ u:\Omega \subset {\mathbb{R}}^n \rightarrow \mathbb{R}^m \;, \: u_i \geq 0 \;, \Omega $ bounded and $ u=g $ on $ \partial \Omega $. This corresponds to a cooperative system, and is a one-phase Bernoulli-type problem. On the other hand, our nonlinearity is of the competitive kind and our problem is a two-phase Bernoulli-type problem.

In \cite{MTV} the functional that is studied is
\begin{equation}\label{DefFunct}
\sum_{i=1}^m \int_{\Omega} \frac{1}{2} |\nabla u_i|^2 + \Lambda \mathcal{L}^n (\cup_{i=1}^m \lbrace u_i \neq 0 \rbrace )dx
\end{equation}
with $ u_i=\phi_i $ on $ \partial \Omega $. This is a two-phase type problem and it is quite close to our functional for $ \alpha =0 $.

The emphasis in these works is on the regularity of the solution and of the free boundary, while the existence of the free boundary is forced by the Dirichlet condition on $ \partial \Omega$, and is not an issue in that context.

For stating our main results we need some algebraic preliminaries.

A \textit{reflection point group} $G $ is a finite subgroup of the orthogonal group whose elements $ g $ fix the origin. We will be assuming for simplicity that $ m=n$ (the general case is presented in \cite{AFS}, Chapter 7), and that $ G $ acts both on the domain space $ \mathbb{R}^n $ and the target space $ \mathbb{R}^m $. A map $ u: \mathbb{R}^n \rightarrow \mathbb{R}^n $ is said to be \textit{equivariant} with respect to the action of $ G $, simply equivariant, if
\begin{align*}
u(gx) = g u(x) \;\;\;,\; \forall \; g \in G \;,\: x \in \mathbb{R}^n
\end{align*}

A \textit{reflection} $ \gamma \in G $ is a map $ \gamma : \mathbb{R}^n \rightarrow \mathbb{R}^n $ of the form
\begin{align*}
\gamma x = x - 2(x \cdot n_{\gamma} ) n_{\gamma} \;\:, \; for \;\: x \in \mathbb{R}^n
\end{align*}
for some unit vector $ n_{\gamma} \in \mathbb{S}^{n-1} $ which aside from its orientation is uniquely determined by $ \gamma $. The hyperplane
\begin{align*}
\pi_{\gamma} = \lbrace x \in \mathbb{R}^n : x \cdot n_{\gamma} =0 \ \rbrace
\end{align*}
is the set of the points that are fixed by $ \gamma $. The open half space $ \mathcal{S}_{\gamma}^+ = \lbrace x \in \mathbb{R}^n : x \cdot n_{\gamma} >0 \rbrace $ depends on the orientation of $ n_{\gamma} $. We let $ \Gamma \subset G $ denote the set of all reflections in $ G $. Every finite subgroup of the orthogonal group $ O(\mathbb{R}^n) $ has a \textit{fundamental region}, that is a subset $ F \subset \mathbb{R}^n $ with the following properties:
$ \\ \\ $
1. $ F$ is open and convex, $ \\ $
2. $ F \cap gF = \emptyset $ for $ I \neq g \in G $, where $ I $ is the identity, $ \\ $
3. $ \mathbb{R}^n = \cup \lbrace g \overline{F} : g \in G \rbrace. \\ $

The set $ \cup_{\gamma \in \Gamma} \pi_{\gamma} $ divides $ \mathbb{R}^n \setminus \cup_{\gamma \in \Gamma} \pi_{\gamma} $ in exactly $ |G| $ congruent conical regions. Each one of these regions can be identified with the fundamental region $ F $ for the action of $ G $ on $ \mathbb{R}^n $. We assume that the orientations of $ n_{\gamma} $ are such that $ F \subset \mathcal{S}_{\gamma}^+ $ and we have
\begin{align*}
F = \cap_{\gamma \in \Gamma} \mathcal{S}_{\gamma}^+ 
\end{align*}

Given $ a \in \mathbb{R}^n $, the \textit{stabilizer} of $ a $, denoted by $ G_a \subset G $ is the subgroup of the elements $ g \in G $ that fix $ a $:
\begin{align*}
G_{a} = \lbrace g \in G : ga =a \rbrace .
\end{align*}

We now introduce two more hypotheses: $ \\ \\ $
\textbf{(H2)}(symmetry) The potential $ W $ is invariant under a reflection (point) group $ G $ acting on $ \mathbb{R}^n $, that is
\begin{align*}
W(gu) = W(u) \;\;\; \textrm{for all} \;\: g \in G \;\: \textrm{and} \;\: u \in \mathbb{R}^n.
\end{align*}
Moreover we assume \eqref{liminfW}.
$ \\ \\ $
\textbf{(H3)}(Location and number of global minima) Let $ F \subset \mathbb{R}^n $ be a fundamental region of $ G $. We assume that $ \overline{F} $ contains a single global minimum of $ W $ say $ a_1 \neq 0 $, and let $ G_{a_1} $ be the stabilizer of $ a_1 $. Setting $ D:= Int(\cup_{g\in G_{a_1}} g \overline{F}) \;,\: a_1 $ is also the unique global minimum of $ W $ in the region $ D $.
\bigskip

Notice that, by the invariance of $ W $, Hypothesis \textbf{(H3)} implies that the number of minima of $ W $ is
\begin{align*}
N = \frac{|G|}{|G_{a_1}|} ,
\end{align*}
where $ |\cdot| $ stands for the number of elements.

We can now state our first main result.
$ \\ $

\begin{theorem}\label{ThEquivMin} ($ 0 < \alpha <2 $) Under hypothesis \textbf{(H1)}-\textbf{(H3)}, there exists an equivariant minimizer $ u $ of $ J $, $ u : \mathbb{R}^n \rightarrow \mathbb{R}^n $, such that $ \\ $
1. $ |u(x) - a_1| =0 $ for $ x \in D $ and $ d(x, \partial D) \geq d_0 \: , $ where $ d_0 $ a positive constant depending on $ ||u||_{L^{\infty} (\mathbb{R}^n, \mathbb{R}^n)} \: , \;\: C^* \; \textrm{and} \; \alpha \;(d_0 \rightarrow + \infty \;\: \textrm{as} \;\: \alpha \rightarrow 2). $
$ \\ $ 2. $ u( \overline{F} ) \subset \overline{F} \;,\: u(\overline{D}) \subset \overline{D} $ (positivity). $ \\ $

\end{theorem}

Hence by equivariance the statements above hold for all $ a_i \:,\; i=1,...,N $, in the respective copy of $ D $.
$ \\ $

\begin{remark}\label{RmkforTh}
In \cite{AGZ} it is shown that $ u \in C^{2,\alpha-1}_{loc} $ for $ \alpha \in (1,2) \;,\; u \in C^{1,\gamma}_{loc} $ for any $ \gamma \in (0,1) $ and $ u \in C^{1,\frac{\alpha}{1-\alpha}} $ for $ \alpha \in (0,1) $. The regularity for $ \alpha \in (0,1) $ is optimal. In Lemma \ref{LemHC} we establish the (suboptimal) estimate $ |u|_{C^{\beta}} < \infty $ (any $ \beta \in (0,1)) $ that holds for all $ \alpha \in [0,2) $ which is sufficient for our purposes. We revisit this point also later. 

The analog of Theorem \ref{ThEquivMin} for $ \alpha =2 \;,\: W \in C^2 $ was established in a series of papers by the first author and G.Fusco. It can be found in \cite{AFS} (Theorem 6.1) where detailed references are given. The main difference with Theorem \ref{ThEquivMin} above is that the condition $ |u(x) -a_1|=0 $ for $ x \in D \;,\: d(x,D) \geq d_0 $, is replaced by $ |u(x) -a_1| \leq K e^{-k d(x, \partial D)} \;,\: x \in D $, where $ k \;,\: K $ are positive constants. In that context the minimizer $ u $ is a classical solution of \eqref{ELequation} while in the present context $ u $ is a weak $ W_{loc}^{1,2} $ solution of \eqref{ELequation} in the complement of the free boundary $ \partial \lbrace u(x) \notin A \rbrace $. The theorem in the smooth case is utilized in our proof of Theorem \ref{ThEquivMin} where we are constructing a minimizer with the positivity property via a $ C^2 $ regularization of the potential. We thus bypass the gradient flow argument used in the proof of the $ \alpha =2 $ case in \cite{AFS} that would be problematic in the present setting. The role of positivity can be seen in the following proposition, which does not presuppose symmetry.
\end{remark}

\begin{proposition}\label{PointwiseEst}($ 0 < \alpha <2 $)
(i) Assume that $ W $ as in (\textbf{H1}) above, and $ u $ a bounded minimizer of $ J \;, \: u: \mathbb{R}^n \rightarrow \mathbb{R}^m \;,\; ||u||_{L^{\infty} (\mathbb{R}^n, \mathbb{R}^m)}< \infty $. Moreover, let $ \mathcal{O} \subset \mathbb{R}^n $ open, assume that 
\begin{equation}\label{PointwiseEsteq1}
d(u( \mathcal{O} ), \lbrace W=0 \rbrace \setminus \lbrace a \rbrace ) \geq k >0
\end{equation}
$ d $ the Euclidean distance, k constant. 

Then given $ q \in (0, ||u||_{L^{\infty} (\mathbb{R}^n, \mathbb{R}^m)}) \;,\: \exists \;\: r_q >0 $ such that
\begin{equation}\label{PointwiseEsteq2}
B_{r_q}(x_0) \subset \mathcal{O} \Rightarrow | u(x_0) -a | <q
\end{equation}
(ii) Let further $ 0 < 2q \leq {\rho}_0 $ (cfr (\textbf{H1})).
Then there exists an explicit constant $ \hat{C} = \hat{C}(\alpha,n) >0 $ (see \eqref{PointwiseEsteq19} , $ \lim_{\alpha \rightarrow 2} \hat{C}(\alpha,n) = \infty $ , $ \lim_{\alpha \rightarrow 0} \hat{C}(\alpha,n) = \infty $ ), such that
\begin{equation}\label{PointwiseEsteq3}
B_{\hat{C} q^{- \alpha} (x_0)} \subset \mathcal{O} \Rightarrow u(x) \equiv a \;\;,\;in\;\: B_{\frac{\hat{C}}{2} q^{- \alpha} (x_0)}
\end{equation}
\end{proposition}

\begin{remark}\label{RmkPointwiseEst} Part (i) of Proposition \ref{PointwiseEst} holds for $ \alpha =2 $, and is a result obtained in \cite{F1}. It can be found also in \cite{AFS} Theorem 5.3. Note that positivity allows the application of this with $ \mathcal{O} =D $, since the solution in $ D $ stays away from all the minima except one. This reveals the nature of \textbf{(H3)}.
$ \\ $
Part (ii) is utilizing a ``Dead Core'' estimate (Lemma \ref{LemDC1} below) which shows that for a function $ v \in W^{1,2}(B_R(x_0)) $
\begin{equation}\label{RmkPointwiseEsteq1}
\begin{cases} \Delta v \geq c^2 v^{\frac{\alpha}{2}} \;\;\;,\;\textrm{weakly in} \;\: W^{1,2} (B_R(x_0)) \\ 0 \leq v \leq \delta \;\;,\; \delta >0 \;\: \textrm{sufficiently small depending on} \; c
\end{cases}
\end{equation}
\end{remark}
Then if
\begin{equation}\label{RmkPointwiseEsteq2}
\begin{cases} dist(y_0, \partial B_R(x_0)) > R_0 \Rightarrow \\
v(y_0) =0 \; \textrm{for} \; R>R_0 = \frac{\sqrt{n(n+2)}}{(1- \frac{\alpha}{2})c} \delta^{\frac{2-\alpha}{4}} \;,\: \alpha \in (0,2)
\end{cases}
\end{equation}
``Dead Core'' regions are sets where the solution is constant.

The first appearance of such a situation was in \cite{CS}, \cite{PS1}, followed by more in depth study in \cite{Sperb}.

$ \\ $

\begin{proposition}\label{PropLowerBd}
($ \alpha =0 $) Let
\begin{equation}\label{PropLowerBdeq1}
J(u) = \int (\frac{1}{2}|\nabla u|^2 + \chi_{A^c}(u))dx
\end{equation}
where $ A := \lbrace W=0 \rbrace = \lbrace a_1,...,a_N \rbrace \subset \mathbb{R}^m \;(N \geq 2),\; A^c = \mathbb{R}^m \setminus A . $ Let $ u $ be a nonconstant minimizer, $ u:\mathbb{R}^n \rightarrow \mathbb{R}^m \;,\; ||u||_{L^{\infty} (\mathbb{R}^n, \mathbb{R}^m)} < \infty $. Suppose that for some $ a_i \in A $ we have
\begin{equation}\label{PropLowerBdeq2}
d(u(B_R(x_0)), \lbrace W=0 \rbrace \setminus a_i ) >0
\end{equation}
Then
\begin{equation}\label{PropLowerBdeq3}
\mathcal{L}^n (\lbrace u = a_i \rbrace \cap B_R(x_0)) \geq c R^n \;,\; R \geq R_0
\end{equation}
for some constant $ c>0 $ independent of $ R $.
\end{proposition}

What about existence of minimizer defined on $ \mathbb{R}^n $ possessing a free boundary and without any symmetry assumptions? This is a difficult open problem for the coexistence of three or more phases. We have the following simple result in this direction.

\begin{proposition}\label{PropFinitePer}
($ \alpha =0 $) Consider the functional
\begin{equation}\label{PropFinitePereq1}
J(u) = \int (\frac{1}{2} | \nabla u|^2 + \chi_{A^c}(u))dx
\end{equation}
where $ A= \lbrace a_1,...,a_N \rbrace $ distinct points in $ \mathbb{R}^m \;,\: A^c = \mathbb{R}^m \setminus A . \\ $ Let $ u: \mathbb{R}^n \rightarrow \mathbb{R}^m $ be a nonconstant minimizer with $ || u ||_{L^{\infty} (\mathbb{R}^n , \mathbb{R}^m)} < \infty $ and $ x_0 \in \mathbb{R}^n $, arbitrary and fixed. Then there exist an $ R_0 >0 $ and at least two distinct points $ a_i \neq a_j $ in $ A $, such that the following estimates hold:
\begin{equation}\label{PropFinitePereq2}
{\mathcal{L}}^n (\overline{B_R(x_0)} \cap \lbrace u(x) = a_k \rbrace) \geq c_k R^n \;, \; R \geq R_0 \;,\; k=i,j
\end{equation}
\begin{equation}\label{PropFinitePereq3}
|| \partial \lbrace u(x) = a_k \rbrace|| (B_R(x_0)) \; \geq \hat{c}_k R^{n-1} \;\;,\; R \geq R_0 \;,\;  k=i,j
\end{equation}
where $ c_k \:, \: \hat{c}_k $ are positive constants, independent of $ x_0 $ and $ R $ (but depending on $ u $). $ || \partial E || $ stands for the perimeter measure of the set $ E $ and $ || \partial E|| (B_R(x_0)) $ denotes the perimeter of $ E $ in $ B_R(x_0) $ (see for instance \cite{EG}).
\end{proposition}

\begin{remark}\label{RmkPrFP}
Proposition \ref{PropFinitePer} holds for the whole range of potentials $ 0< \alpha <2 $ defined in \textbf{(H1)} but with a significantly harder proof \cite{AGZ}.
\bigskip

The natural way of constructing entire solutions $ u $ to \eqref{ELequation} without symmetry requirements is by minimizing over balls $ B_R$ with appropriate boundary conditions forcing the phases on $ B_R $:
\begin{align*}
\min J_{B_R} (v) \;\;,\; v = g_R \;\;,\; \textrm{on} \;\: \partial B_R,
\end{align*}
and taking the limit along subsequences of minimizers $ u_R $
\begin{align*}
u = \lim_{R \rightarrow \infty} u_R
\end{align*}
\end{remark}
\begin{remark}\label{RmkSymmetrica=0} The result from Proposition \ref{PropFinitePer} holds for the symmetric case as in Theorem \ref{ThEquivMin} for $ \alpha=0 $, and provides some quantitative information on the Dead Core. We have not been able to establish the exact analog of Theorem \ref{ThEquivMin} for $ \alpha=0 $.
\end{remark}
\begin{proposition}\label{PropSymmetrica=0} ($ \alpha =0 $)
Under the hypothesis \textbf{(H1)}-\textbf{(H3)} and $ N=m+1 $, there exist a nontrivial equivariant minimizer of $ J(u) = \int ( \frac{1}{2}| \nabla u|^2 + \chi_{\lbrace u \in S_A \rbrace} )dx \;\:,\; u : \mathbb{R}^n \rightarrow \mathbb{R}^n $, such that
1. $ u( \overline{F}) \subset \overline{F} \;\:,\; u(\overline{D}) \subset \overline{D} $ (positivity). $ \\ $
2. $ \mathcal{L}^n ( D_R \cap \lbrace u =a_1 \rbrace ) \geq c R^n \;\;, \; R \geq R_0 $, where $ D_R = D \cap B_R(0) $ ($ D $ from \textbf{(H3)}). $ \\ $
3. $ \mathcal{L}^n (D_R \cap \lbrace u \neq a_1 \rbrace) \leq C R^{n-1} \;\;,\; R \geq R_0 $.
\end{proposition}

$ \\ $

A convenient hypothesis guaranteeing $ ||u||_{L^{\infty}} < \infty $ is \footnote{For $ \Omega \subset \mathbb{R}^n $ open, by linear elliptic theory $ u \in C^2(\Omega ; \mathbb{R}^m) $. Set $ v= |u|^2 $, then $ \Delta v = 2 W_u(u) \cdot u + 2 |\nabla u|^2 >0 $, for $ u >M $. Hence $ \max |u|^2 \leq M $ if $ v $ attains its max in the interior of $ \Omega $.}
\begin{equation}
\begin{cases} W_u(u) \cdot u \geq 0 \;\;\;,\; \textrm{for} \;\; |u| \geq M \:,\; \textrm{some} \; M \\ | g_R| \leq M
\end{cases}
\end{equation}

The existence of one-dimensional minimizers ($ u: \mathbb{R} \rightarrow {\mathbb{R}}^n \;\:,\; $ i.e. connections) for $ \alpha \in (0,2) $, can be obtained by Theorem 2.1, p.34 in \cite{AFS}. For the $ \alpha =0 $ case, where $ W $ is a characteristic function, one-dimensional minimizers are affine maps connecting the phases. More precisely,
\begin{equation}
u(x) = \begin{cases} a_1  \;\;\;, \;\; x < - L
\\ a_2 \;\;\;, \;\; x >L \\ \frac{a_2-a_1}{2L}x + \frac{a_1 +a_2}{2} \;\;,\; x \in [-L,L]
\end{cases} 
\end{equation}
and by minimality one can see that $ L = \frac{|a_2 -a_1|}{2\sqrt{2}} $, which is formally what we expect from the free boundary condition $ |\nabla u|^2 = 2 $ (see \eqref{formal0}).

The basic question of course is whether a nontrivial minimizer $ u $ connecting the phases can be constructed. We know from the work on the De Giorgi referred above conjecture that for $ m=1 $, and in low dimensions, any such minimizer will depend on a single variable, and so in a sense is trivial. For the system we expect otherwise, and indeed this was shown to be the case in the equivariant setting and for smooth potentials, in the book \cite{AFS}.

There are a few tools that we utilize in the sequel that because of their independent interest we mention explicitly.
$ \\ $

\underline{The Basic Estimate} $ \\ $
For minimizers, $ 0 \leq \alpha <2 $ satisfying $ |u(x)| \leq M \;,\; x \in \mathbb{R}^n $ we have that there exists $ r_0 >0 $ such that for any $ x_0 \in \mathbb{R}^n $
\begin{equation}\label{TheBasicEstimate}
J_{B_r(x_0)} (u) \leq C_0 r^{n-1} \;,\; r \geq r_0 >0\;,
\end{equation}
$C_0 >0$ constant, independent of $ u $, but depending on $ M $.
\bigskip

For $ \alpha \in [1,2) $ elliptic theory applied to \eqref{ELequation} implies $ ||\nabla u||_{L^{\infty}} < \infty $, and (\ref{TheBasicEstimate}) follows easily (cfr. \cite{AFS} Lemma 5.1). For $ \alpha \in [0,1) $, and $ m=1 $, it is already mentioned in \cite{CC}. We prove it in Lemma \ref{LemBE}. The estimate \eqref{TheBasicEstimate} is utilized in the proof of Proposition \ref{PropFinitePer}, and also in the proof of Proposition \ref{PointwiseEst} on which Part 1 of Theorem \ref{ThEquivMin} is based. Finally \eqref{TheBasicEstimate} is also utilized in the proof of the Density Estimate that we discuss below.
$ \\ $

\underline{The Density Estimate} $ \\ $
For minimizers $ u $ of the functional $ J $ in \eqref{DefJu}, $0 \leq \alpha <2$ satisfying $ |u(x)| \leq M $, we have
\begin{equation}\label{TheDensityEstimate}
\begin{cases} \mathcal{L}^n (B_{r_0}(x_0) \cap \lbrace |u-a|> \lambda\rbrace) \geq \mu_0 >0 \Rightarrow \\ \mathcal{L}^n (B_r(x_0) \cap \lbrace |u-a|> \lambda\rbrace) \geq Cr^n \;\;,\; r \geq r_0
\end{cases}
\end{equation}
$ C=C(\mu_0,\lambda) . \\ $

This is an important estimate of Caffarelli and Cordoba \cite{CC} established in the scalar case $ m=1$, and extended to the vector case by the first author and G.Fusco. We refer to \cite{AFS} Theorem 5.2, where detailed references can be found. The proof in \cite{AFS} has a gap for $ 0 \leq \alpha <1 $ since it is utilizing \eqref{TheBasicEstimate} that was taken for granted then but proved in the present paper.
$ \\ $

\underline{The H\"older Estimate} $ \\ $
For minimizers $ u $ of the functional $ J $ in \eqref{DefJu}, $ 0 \leq \alpha <2 $, satisfying $ |u(x)| < M \;,\; x\in \mathbb{R}^n $, we have the estimate
\begin{equation}\label{TheHolderEstimate}
| u(x) - u(y) | \leq C | x-y | \: \textrm{ln}(|x-y|^{-1}) \;\;\;,\; \forall \; x,y \;,\; |x-y| \leq \frac{1}{2}
\end{equation}
which implies $ u \in C^{\beta}(\mathbb{R}^n, \mathbb{R}^m) \;,\; \forall \beta \in (0,1) \;, \; C=C(M), $ that has already be mentioned. $ \\ $

This is established in \cite{ACF} for $ m=1 $ and $ \alpha=0$. We give a detailed proof in Lemma \ref{LemHC}. It is utilized in several places. For example in establishing Proposition \ref{PointwiseEst} (i) we proceed by a contradiction argument that invokes the Density Estimate. Here uniform continuity is essential, and is provided by \eqref{TheHolderEstimate}. It is also instrumental for the derivation of the Basic Estimate \eqref{TheBasicEstimate}. $ \\ $

The H\"older continuity is also needed in the proof of the Containment result presented in Appendix A, that we now describe.
$ \\ $

\underline{The Containment} $ \\ $
This states that for the special potentials
\begin{equation}
W(u) = \begin{cases} W^{\overline{\alpha}}(u) := \prod_{k=1}^{m+1} |u-a_k|^{\alpha_k} \;\;,\; \overline{\alpha} = (a_1,...,a_{m+1}) \;,\; 0 < a_k <2 \\ W^0(u) := \chi_{A^c}(u) \;\;,\; A= \lbrace a_1,...,a_{m+1} \rbrace
\end{cases}
\end{equation}
where the vectors $ \lbrace a_2-a_1,...,a_{m+1}-a_1 \rbrace $ are linearly independent in $ \mathbb{R}^m $, critical points of $ J(u) = \int (\frac{1}{2} |\nabla u|^2 + W(u))dx \;\;,\; u: \mathbb{R}^n \rightarrow \mathbb{R}^n \;,\; |u(x)| <M $, map $ \mathbb{R}^n $ inside the closure of the convex hull of $ A \;,\; \overline{co}(A) $. This result was obtained jointly by the first author and P.Smyrnelis, in unpublished work. Its proof requires uniform continuity, and so for $ \alpha \in (0,1) $ we need to restrict ourselves to minimizers for which \eqref{TheHolderEstimate} holds.

This result shows that $ J^0 $ is in some natural way the limit of $ J^{\overline{\alpha}} $, as $ \overline{\alpha} \rightarrow 0 $, and actually we establish a $ \Gamma $-limit type relationship in Lemma \ref{LemJaLim}. $ \\ $ 

This paper is structured as follows. $ \\ $
In section 2 we state and prove various Lemmas already mentioned in the introduction. $ \\ $
In section 3 we give the proofs of Theorem \ref{ThEquivMin}, Propositions \ref{PointwiseEst}, \ref{PropLowerBd} and \ref{PropFinitePer}. $ \\ $
In Appendix \ref{sec:appendixA} we state and prove the containment result, and in Appendix \ref{sec:appendixB} we give a formal argument, taken essentially from \cite{AFS}, that explains the free boundary conditions in \eqref{formal(0,2)} and \eqref{formal0}.

$ \\ \\ $

\par\noindent{\bf Acknowledgements}
We are greateful to Panayotis Smyrnelis for his interest in this work and his numerous comments that improved the paper. AZ would like to thank Prof. Luc Nguyen for pointing out the log-estimate argument  in the proof of Lemma $2.1$. Finally we would like to thank Zhiyuan Geng for introducing us to free boundary problems.

The work of A.Z. is supported by the Basque Government through the BERC 2018-2021
program, by Spanish Ministry of Economy and Competitiveness MINECO through BCAM
Severo Ochoa excellence accreditation SEV-2017-0718 and through project MTM2017-82184-
R funded by (AEI/FEDER, UE) and acronym ``DESFLU".

D.G. would like to acknowledge support of this work by the project ``Innovative Actions in Environmental Research and Development (PErAn'' (MIS 5002358) which is implemented under the ``Action for the Strategic Development on the Research and Technological Sector'', funded by the Operational Programme ``Competitiveness, Entrepreneurship and Innovation'' (NSRF 2014-2020) and co-financed by Greece and the European Union (European Regional Development Fund).

N.D.A. held a BCAM visiting fellowship in the fall of 2019 during which some of the results of the present paper were established; also would like to thank his host, Arghir Zarnescu and the people in the institute for their hospitality.

$ \\ \\ $

\section{Basic Lemmas}

\subsection{Regularity of $ u $ }
$ \\ $
We will prove a logarithmic estimate for bounded minimizers, following closely the proof of Theorem 2.1 in \cite{ACF} (see also Lemma 2 in \cite{CSY}). We have: $ \\ $

\begin{lemma}\label{LemHC} ($ 0 \leq \alpha <2 $, H\"older Continuity)
$ \\ $ Let $ u : {\mathbb{R}}^n \rightarrow {\mathbb{R}}^m $ a minimizer of $ J \;,\: | u(x) | < M \;,\: W $ satisfying \textbf{(H1)} for $ 0< \alpha <2 $ and $ W = \chi_{A^c}(u) $ for $ \alpha = 0 $. Then there exists constant $ C=C(M) $, such that
\begin{equation}\label{Lem21eq1}
| u(x) - u(y) | \leq C | x-y | \textrm{ln}(|x-y|^{-1}) \;\;\;,\; \forall \; x,y \;,\; |x-y| \leq \frac{1}{2}
\end{equation} 
In paricular, $ u \in C^{\beta} ({\mathbb{R}}^n ; {\mathbb{R}}^m) \;,\; \forall \: \beta \in (0,1) . \\ $
\end{lemma}
\begin{proof}
We restrict ourselves to $ 0 \leq \alpha <1 $, since the result follows immediately for $ \alpha \in [1,2] $ by linear elliptic theory. We begin with the case $ 0<\alpha<1 . $

For an arbitrary $B_r(x_0) $ let $v_r$ be the harmonic function equal to $u$ on $\partial B_r$.  Then by the maximum principle $v_r$ is also bounded and taking into account the specific form of the potential \eqref{Walpha} we have that there exists an $ M$ such that:
\begin{equation}\label{Lem21eq2}
|u(x)|,|v_r(x)|,|W^{\overline{\alpha}}(u(x))|,|W^{\overline{\alpha}}(v_r(x))|\le M \;,\; \forall x\in B_r(x_0), \; \alpha\in [0,1]
\end{equation}
$ \\ $
Then using the minimality of $u$ and the non-negativity of the potentials $W^{\overline{\alpha}}$ together with \eqref{Lem21eq2} we have: $ \\ $

\bea\label{Lem21eq3}
\int_{B_r} |\nabla u(x)|^2\,dx&\le \int_{B_r} |\nabla u(x)|^2+W^{\overline{\alpha}}(u(x))\,dx\le \int_{B_r} |\nabla v_r(x)|^2+W^{\overline{\alpha}}(v_r(x))\,dx\non\\
&\le M|B_r|+\int_{B_r} |\nabla v_r(x)|^2\,dx
\eea 
hence
\begin{equation}\label{Lem21eq4}
\int_{B_r}|\nabla u(x)|^2-|\nabla v_r(x)|^2\,dx\le Cr^n
\end{equation}

On the other hand we have:

\bea\label{Lem21eq5}
\int_{B_r}|\nabla u(x)|^2-|\nabla v_r(x)|^2\,dx&=\int_{B_r}(\nabla u(x)+\nabla v_r(x),\nabla u(x)-\nabla v_r(x))\,dx\non\\
&=\int_{B_r}|\nabla (u(x)-v_r(x))|^2\,dx+2\int_{B_r}(\nabla u-\nabla v_r)\nabla v_r\,dx\non\\
&=\int_{B_r}|\nabla (u(x)-v_r(x))|^2\,dx
\eea
where for the last inequality we used that $v_r$ is harmonic and equal to $u$ on $\partial B_r$.

Thus we get: 
\be\label{Lem21eq6}
\int_{B_r} |\nabla (u(x)-v_r(x))|^2\,dx\le Cr^n
\ee
From the previous estimate, it suffices to show that
\begin{equation}\label{Lem21eq7}
\int_{B_s} | \nabla u|^2 \leq C s^n [ \: \textrm{ln}^2(r/s) + 1]
\end{equation}
This would imply \eqref{Lem21eq1}. $ \\ $
To prove \eqref{Lem21eq7}, we proceed as follows:
\begin{align*}
\int_{B_s} | \nabla u |^2 \leq \int_{B_s} | \nabla v_{2s} |^2 + \int_{B_s} | \nabla (u-v_{2s} )| | \nabla (u+ v_{2s} )|
\end{align*}

The first integral on the right side is estimated using the subharmonicity of $ | \nabla v_{2s} |^2 $, and then the minimality of $ v_{2s}$. So, 
\begin{align*}
\frac{1}{|B_s|} \int_{B_s}| \nabla v_{2s}|^2 \leq \frac{1}{|B_{2s}|} \int_{B_{2s}} | \nabla v_{2s}|^2 \leq \frac{1}{|B_{2s}|} \int_{B_{2s}}  |\nabla u|^2
\end{align*}
by \eqref{Lem21eq5}.

The second integral is estimated by enlarging the domain to $ B_{2s}$, then Cauchy-Schwartz, the established bound and the minimality of $ v_{2s} $
\begin{align*}
\begin{gathered}
\frac{1}{|B_s|} \int_{B_s}| \nabla (u-v_{2s})| | \nabla (u+ v_{2s} )| \leq \\ \frac{|B_{2s}|}{|B_s|}(\frac{1}{|B_{2s}|} \int_{B_{2s}}| \nabla (u-v_{2s})|^2)^{\frac{1}{2}} ( \frac{2}{|B_{2s}|} \int_{B_{2s}} | \nabla u|^2 + | \nabla v_{2s}|^2 )^{\frac{1}{2}} \leq C (\frac{1}{|B_{2s}|} \int_{B_{2s}} | \nabla u|^2 )^{\frac{1}{2}}
\end{gathered}
\end{align*}
by \eqref{Lem21eq5}, \eqref{Lem21eq6}.

So if we set
\begin{align*}
x_k = \frac{1}{|B_{2^{-k}}|} \int_{B_{2^{-k}}}| \nabla u|^2
\end{align*}
then
\begin{align*}
x_{k+1} \leq x_k + Cx_k^{1/2}
\end{align*}
Induction gives
\begin{align*}
x_{k+1} \leq C' k^2
\end{align*}
from which you have \eqref{Lem21eq7}.
$ \\ $

Estimate \eqref{Lem21eq1} then follows from the proof of Morrey's embedding. Indeed, suppose $ x $ and $ y $ are given, of distance $ 2s $ apart. Let $ z $ be the midpoint. Then, by mean value theorem,
\begin{align*}
\frac{1}{|B_s|} \int_{B_s}| u(x) - u(p)| dp \leq Cs \frac{1}{|B_s|} \int_{B_s} \int_0^1 | \nabla u(p +t(x-p) ) | dt dp
\end{align*}
Thus, interchanging the order of integration and using \eqref{Lem21eq7}, we get
\begin{align*}
\frac{1}{|B_s|} \int_{B_s}| u(x) - u(p)| dp \leq Cs[\textrm{ln}(1/s) +1]
\end{align*}
The estimate for $ | u(x)-u(y)| $ then follows from triangle inequality. $ \\ $

The proof for the case $ \alpha=0 $ is similar, the only difference being that instead of the bound in \eqref{Lem21eq2} $ |W^0(u(x))| , |W^0(v_r(x))| \leq 1 $ is used.
\end{proof}

$ \\ \\ $
\subsection{The Basic Estimate:}
$ \\ $

\begin{lemma}\label{LemBE} Let $ u: {\mathbb{R}}^n \rightarrow {\mathbb{R}}^m $ minimizer of $ J \;,\: | u(x)| <M \;,\: W $ satisfying \textbf{(H1)} for $ 0< \alpha <2 $ and $ W=W^0 $ for $ \alpha =0 $. Then there is a constant $ C_0 =C_0(W,M) $ independent of $ x_0 $ and such that
\begin{align*}
J_{B_r (x_0)} (u) \leq C_0 r^{n-1} \;\;\;,\; r>r_0
\end{align*}
\end{lemma}
\begin{proof} $ \\ $
1. For $ \alpha \in [1,2) $, utilizing elliptic estimates we obtain $ | \nabla u(x)| < C(M) \;,\; x \in {\mathbb{R}}^n .$ The estimate then follows by constructing a competitor $ v(x) $ on a ball via
\begin{align*}
v(x) = \begin{cases} a \;\;\;\;\;\;\;\;\;\;,\;\;|x-x_0| \leq r-1 \\
(r- |x-x_0|)a + (|x-x_0| - r+1) u(x) \;\;, \; r-1 < | x-x_0| \\
u(x) \;\;\;\;\;,\;\; |x-x_0 |>r
\end{cases}
\end{align*}
and utilizing the minimality of $ u $ (cfr Lemma 5.1 \cite{AFS}). Here we can take $r_0=0. \\ \\ $
2. For $ \alpha \in (0,1) $, we aim to prove the estimate: $ \\ $
\end{proof}

\begin{lemma}\label{LemJRn-1}
Let $u:\R^n\to\R^m$ be a bounded local minimizer for  the energy functional $J$ in \eqref{DefJu} with the potential $W_\alpha$ as in \textbf{(H1)}. Then there exists  constant  $C,R_0>0$ independent of $u$ such that:
\be\label{est:JRn-1}
J(u; A(R))\le  CR^{n-1}, \forall R\ge R_0
\ee where $C$ is independent of $R\ge R_0$ and $A(R):=B_R(x_0)\setminus B_{R-1}(x_0)$.
\end{lemma}

\begin{proof}
We first claim  that there exists a constant $\tilde C>0$ such that for any $x_0\in \R^n$ we have, for $u$ a bounded local minimizer:

\be\label{est:unifballgradient}
\int_{B_1(x_0)}|\nabla u(x)|^2 dx \le \tilde C
\ee 

To this end we consider the function $v\in W^{1,2}(B_1(x_0))$ with $v=u$ on $\partial B_1(x_0)$ and $\Delta v=0$ in $B_1(x_0)$. Since 
$u$ is bounded, by the maximum principle we have that $v$ is also bounded and taking into account the hypothesis \textbf{(H1)} for the potential $W_\alpha$ we have that there exists $M>0$ such that:

\be\label{est:bdduvW}
|u(x)|,|v(x)|,W_\alpha(u(x)),W_\alpha(v(x))\le M,\forall x\in\R^n,\alpha\in [0,1]
\ee

We then have:

\bea\label{LemJRn-1eq1}
\int_{B_1(x_0)}|\nabla u(x)|^2\,dx&\le \int_{B_1(x_0)} |\nabla u(x)|^2+W_\alpha(u(x))\,dx\le  \int_{B_1(x_0)} |\nabla v(x)|^2+W_\alpha(v(x))\,dx\non\\
&\le M|B_1|+ \int_{B_1(x_0)}|\nabla v(x)|^2\,dx=M|B_1|+ \int_{\partial B_1(x_0)} \frac{\partial v}{\partial \nu} v\,d\sigma\non\\
&\le M|B_1|+\| \frac{\partial v}{\partial \nu}\|_{H^{-\frac 12}(\partial B_1(x_0))}\|v\|_{H^{\frac 12}(\partial B_1(x_0))}\non\\
&\le M|B_1|+C\|\nabla v\|_{L^2( B_1(x_0))}\|v\|_{H^{\frac 12}(\partial B_1(x_0))}\non\\
&\le M|B_1|+\frac{1}{2}\|\nabla v\|_{L^2( B_1(x_0))}^2+C\|v\|_{H^{\frac 12}(\partial B_1(x_0))}^2\non\\
&= M|B_1|+\frac{1}{2}\|\nabla v\|_{L^2( B_1(x_0))}^2+C\|u\|_{H^{\frac 12}(\partial B_1(x_0))}^2
\eea where for the first inequality we used the non-negativity of $W_\alpha$, for the second the  local minimality of $u$, and for the third   the estimates \eqref{est:bdduvW}.  For the first equality we used the fact that $v$ is a harmonic function and an integration by parts, while for the last equality we used that $u=v$ on $\partial B_1(x_0)$. For the penultimate inequality we used the continuity of the normal part of  trace operator on the space $L^2_{div}=\{f\in L^2; \textrm{div}\, f\in L^2\}$ (see for instance Prop. $3.47$, (ii) in \cite{CD}).

We obtain thus:

\be\label{LemJRn-1eq2}
\int_{B_1(x_0)}|\nabla u(x)|^2\,dx\le M|B_1|+C\|u\|_{H^{\frac 12}(\partial B_1(x_0))}^2
\ee

On the other hand we have (see for instance \cite{NPV}):

\be\label{LemJRn-1eq3}
\|u\|_{H^{\frac 12}(\partial B_1(x_0))}^2=\int_{\partial B_1(x_0)}\int_{\partial B_1(x_0)} \frac{|u(x)-u(y)|^2}{|x-y|^{n-1+1}}\,dx\,dy\le C
\ee where for the last inequality we used the logarithmic estimate \eqref{LemHC}.

Combining the last two estimates we obtain the claimed uniform estimate \eqref{est:unifballgradient}. On the other hand, thanks to 
estimate \eqref{est:bdduvW} we have

\be\label{LemJRn-1eq4}
\int_{A(R)}W_\alpha (u(x))\,dx\le CR^{n-1}
\ee which combined with the fact that one can cover $A(R)$ with $CR^{n-1}$ balls of radius $1$ and estimate \eqref{est:unifballgradient} provides the desired estimate \eqref{est:JRn-1}.

\end{proof}
$ \\ $
\textit{Note:} Lemma 2.3 implies Lemma 2.2 ($ \alpha \in (0,1) $) by considering the comparison function $ v(x) $ as in $ \alpha \in [1,2) $ case.

$ \\ \\ $
\subsection{The ``Dead Core'' estimate:} $ \\ $

Now, we proceed with a useful calculation. From the hypothesis \textbf{(H1)} for $ W $ we have that for $ |u-a|<<1 \:,$ it holds that $ \;\:  W_u (u) \cdot (u-a) \geq \overline{c}^2 | u-a |^{\alpha} $ with $ \overline{c}^2 = \alpha C^* \;,\: \alpha \in (0,2) $. Set $ v(x) = |u-a|^2 $.

Then
\begin{equation}\label{DifIn1}
\begin{gathered}
\Delta v = \sum_{i=1}^n 2 ((u(x)-a) u_{x_i})_{x_i} = 2| \nabla u|^2 + 2 (u(x)-a) \Delta u = \\ 2| \nabla u|^2 + 2 W_u(u) \cdot (u(x)-a) \geq 2| \nabla u|^2 + 2 \overline{c}^2 | u-a |^{\alpha}
\end{gathered}
\end{equation}
Therefore,
\begin{equation}\label{DifIn2}
\Delta v \geq c^2 | u-a |^{\alpha} = c^2 v^{ \frac{\alpha}{2}} \;\;,\; \textrm{where} \; c^2 =2\alpha C^*.
\end{equation}
$ \\ $

\begin{definition}\label{DefDC} Let $ \Omega \subset {\mathbb{R}}^n $ open and $ v \in W^{1,2}_{loc}(\Omega, \mathbb{R}) $, a region $ {\Omega}_0 \subset \Omega $ is called a \textit{dead core} if $ v \equiv 0 $ in $ {\Omega}_0 . $
\end{definition}
$ \\ $

For the convenience of the reader, let us now state some results from \cite{Sperb}. 
$ \\ \\ $
The article \cite{Sperb} is concerned with the problem

\begin{equation}\label{eq:Sperb1}
\begin{cases} \Delta u = c^2 u^p \;\: \textrm{in} \; \Omega \subset {\mathbb{R}}^n \\
u=1 \;\: \textrm{on} \; \partial \Omega
\end{cases}
\end{equation}
with $ p \in (0,1) $. We call that a ``dead core'' $ {\Omega}_0 $ develops in $ \Omega $, i.e. a region where $ u \equiv 0 $.

Let $ X(s) $ be a solution of 
\begin{align}\label{eq:Sperb2}
\begin{cases} X''(s) = c^2 X^p(s) \;\: \textrm{in} \;\: (0,s_0) \\
X'(0) =0 \;\:, X(s_0) =1 \;\: 
\end{cases}
\end{align}

As a first choice of a linear problem consider the ``torsion problem'' , i.e.
\begin{align}\label{eq:Sperb3}
\begin{cases} \Delta \psi +1 =0 \;\:  \textrm{in} \;\: \Omega \\
\psi =0 \;\:  \textrm{on} \;\: \partial \Omega
\end{cases}
\end{align}

One then constructs a supersolution $ \overline{u}(x) $ to \eqref{eq:Sperb1} having the same level lines as the torsion function by setting
\begin{equation}\label{eq:Sperb4}
\overline{u}(x) =X(s(x)) \;, \;\: x \in \Omega
\end{equation}
where
\begin{equation}\label{eq:Sperb5}
s(x) = \sqrt{ 2 ( \psi_m - \psi(x)) } \;\:, \psi_m = \max_{ \Omega} \psi
\end{equation}
In problem \eqref{eq:Sperb2} we choose $ s_0 = \sqrt{2 \psi_m } . \\ 
\\ $
\begin{theorem}\label{ThSperb} (\cite{Sperb}) Assume that the mean curvature of $ \partial \Omega $ is nonnegative everywhere. Then
\begin{equation}\label{eq:Sperb6}
\begin{gathered} \overline{u}(x) = X(s(x)) \;\: is\;\: a \;\: supersolution, \;\: i.e. \\ 
\Delta \overline{u} \leq c^2 {\overline{u}}^p \;\: \textrm{in} \;\: \Omega \\
\overline{u}=1 \;\: \textrm{on} \;\: \partial \Omega
\end{gathered}
\end{equation}
\end{theorem}

One of the corollaries of this Theorem is the information on the location and the size of the ``dead core'' $ {\Omega}_0 $, which may be stated as $ \\ \\ $

\begin{corollary}\label{CorolSperb} (\cite{Sperb}) The dead core $ {\Omega}_0 $ contains the set
\begin{align*}
\lbrace x \in \Omega | \psi (x) \geq d(p,c)[ \sqrt{ 2 \psi_m } -\frac{1}{2} d(p,c) ] \rbrace \:,
\end{align*}
where $ d(p,c) := \dfrac{\sqrt{2(p+1)}}{(1-p)c} . \\ $
\end{corollary}

We will now utilize the above for the proof of the following Lemmas. 
$ \\ $

\begin{lemma}\label{LemDC1} Let $ \Omega = B_R (x_0) \subset {\mathbb{R}}^n $ and $ v \in C^2 ( \Omega ; {\mathbb{R}}_+ ) $ satisfy the following assumptions: 

\begin{align}\label{DCeq1}
\begin{gathered}
\Delta v(x) \geq c^2 v^{\frac{ \alpha}{2}} (x) \;\: ,\; x \in \Omega \\
v(x) \leq \delta \;\; , \; x \in \partial \Omega 
\end{gathered}
\end{align}
$ \alpha \in (0,2) \Leftrightarrow \frac{ \alpha}{2} = p \in (0,1) . \\ $ Then if $ y_0 \in \Omega $ is such that $ dist(y_0, \partial \Omega ) > R_0 \Rightarrow v(y_0) =0 . \\ \\ $
where $ R_0 := \begin{cases} \sqrt{n} d(p, \hat{c}) \;\:, R \geq \sqrt{n} d(p,\hat{c}) \\ 2R - \sqrt{n} d(p,\hat{c}) \;\: , \frac{1}{2} \sqrt{n} d(p,\hat{c}) < R < \sqrt{n} d(p,\hat{c}) \end{cases} . \\ \\ $ and $ d(p, \hat{c}) := \dfrac{\sqrt{2(p+1)}}{(1-p)\hat{c}} \;,\;\: \hat{c} = \dfrac{c}{{\delta}^{\frac{1-p}{2}}}. \\ $
\end{lemma}
\begin{proof}
From the maximum principle we have that $ v(x) \leq \delta $ in $ \Omega \\ $
Define $ \hat{v}:= \dfrac{v}{ \delta} $ and $ \hat{c} := \dfrac{c}{{ \delta}^{\frac{1-p}{2}}} $, then we have:
\begin{align*}
\begin{gathered}
\begin{cases}
\Delta \hat{v} (x) \geq {\hat{c}}^2 \hat{v}^{\frac{ \alpha}{2}} (x) \;\: ,\; x \in \Omega . \\
\hat{v} (x) \leq 1 \;\; , \; x \in \partial \Omega
\end{cases}
\end{gathered}
\end{align*}
For $ \Omega = B_R (x_0) $ we have that 
\begin{equation}\label{DCeq2}
\psi (x) = \frac{R^2}{2n} - \frac{1}{2n} |x - x_0|^2 \;\:,\; \psi_m = \frac{R^2}{2n}
\end{equation}
is a solution to the problem:
\begin{align}\label{DCeq3}
\begin{gathered}
\begin{cases} \Delta \psi (x) +1 =0 , \; x \in \Omega \\
\psi (x) = 0 \;\;, \; x\in \partial \Omega
\end{cases}
\end{gathered}
\end{align}
Also, we have that if:
\begin{equation}\label{DCeq4}
\begin{gathered}
\begin{cases} \Delta u \leq c^2 u^p , \; x \in \Omega \\
\Delta v \geq c^2 v^p , \; x \in \Omega 
\\ v \leq u  \; , \; x \in \partial \Omega
\end{cases}
\end{gathered}
\end{equation}
then $ v \leq u $ , in $ \Omega $. So since $ u,v \geq 0 $, if $ u(x_1)=0 \Rightarrow v(x_1) =0 . \\ $ Such $ u $ is defined in \cite{Sperb} via $ \psi $ in Theorem \ref{ThSperb} (supersolution with $ u=1 \geq \hat{v} $ on  the boundary). Then by Corollary \ref{CorolSperb} in \cite{Sperb}, the dead core of $ \overline{u} $ contains the set  $  \lbrace x \in \Omega | \psi(x) \geq C_0 := d(p,\hat{c})[ \frac{R}{\sqrt{n}} - \frac{1}{2} d(p,\hat{c}) ]  \rbrace $, that is if $ \\ y_0 \in \lbrace \psi(x) \geq C_0 \rbrace \Rightarrow \overline{u}(y_0) =0 $ and thus $ \hat{v} (y_0)  = v(y_0) =0 . \\ $
Since $ \psi $ has the form \eqref{DCeq2} we can see that 
\begin{align*}
\lbrace x \in \Omega | \psi(x) \geq C_0 \rbrace = \lbrace dist(x, \partial \Omega) \geq R_0 \rbrace
\end{align*}
as follows:
\begin{align*}
\begin{gathered}
\psi(x) \geq C_0 \Leftrightarrow \frac{R^2}{2n} - \frac{1}{2n} |x - x_0|^2 \geq C_0 \Leftrightarrow \sqrt{R^2 - 2n C_0} \geq |x - x_0| \\ \Leftrightarrow R - |x - x_0| \geq R - \sqrt{R^2 - 2n C_0}= R - \sqrt{R^2 - 2 \sqrt{n} d(p,\hat{c}) R +n (d(p,\hat{c}))^2  } = \\ = R - | R - \sqrt{n} d(p,\hat{c}) | = R_0
\end{gathered}
\end{align*}
and notice that: $ dist(x, \partial \Omega) = dist(x, \partial B_R (x_0)) = R - dist(x,x_0) $

\end{proof}
$ \\ $
\textit{Notes:} (1) $ \hat{c} $ depends on $ \delta $ and tends to infinity as $ \delta $ tends to zero. $ \\ $
(2) $ d(p, \hat{c}) $ tends to zero as $ \delta $ tends to zero, and so does $ C_0 . \\ $

\begin{remark}\label{RmkDC} If we take $ \tilde{ \Omega} $ open set, such that $ B_R(x_0) \subset \tilde{ \Omega} $ and 
\begin{align*}
\begin{gathered}
\begin{cases} \Delta \tilde{ \psi} (x) +1 =0 , \; x \in \tilde{ \Omega} \\
\tilde{ \psi} (x) = 0 \;\;, \; x\in \partial \tilde{ \Omega}
\end{cases}
\end{gathered}
\end{align*}
then, we have: $ \psi \leq \tilde{ \psi } \Rightarrow \lbrace \psi (x) \geq C_0 \rbrace \subset \lbrace \tilde{ \psi } (x) \geq C_0 \rbrace \Rightarrow \lbrace x \in B_R(x_0) : dist( \partial B_R (x_0),x) \geq R_0 \rbrace \subset \lbrace \tilde{ \psi } (x) \geq C_0 \rbrace . \\ $ Thus, the above theorem holds for more general open sets that contain a ball $ B_R(x_0) $.
\end{remark}
$ \\ $

\begin{lemma}\label{LemDC2} Let $ D $ open, convex $ \subset {\mathbb{R}}^n $ and for some $ d_0 >0 , \\ \Omega := \lbrace x \in D : dist(x, \partial D ) \geq d_0 \rbrace $ and let $ v \in C^2 ( D ; {\mathbb{R}}_+ ) $ satisfying:
\begin{align}\label{DCeq5}
\begin{gathered}
\Delta v(x) \geq c^2 v^{\frac{ \alpha}{2}} (x) \;\: ,\; x \in \Omega \\
v(x) \leq \delta \;\; , \; x \in \Omega 
\end{gathered}
\end{align}
$ \alpha \in (0,2) \Leftrightarrow \frac{ \alpha}{2} = p \in (0,1) . \\ $
Then if $ x_0 \in D $ such that $ dist(x_0, \partial D) \geq d_0 +  2 \frac{ \sqrt{2n(p+1)} }{(1-p)\hat{c}} \Rightarrow v(x_0) =0. \\ $
\end{lemma}
\begin{proof}
We have that:
\begin{align*}
\lbrace x \in D: dist(x, \partial D ) \geq d_0 + 2 \frac{ \sqrt{2n(p+1)} }{(1-p)\hat{c}} \rbrace = \lbrace x \in \Omega : dist(x, \partial \Omega ) \geq 2 \frac{ \sqrt{2n(p+1)} }{(1-p)\hat{c}} \rbrace
\end{align*}
and $ \Omega $ is convex (parallel sets have at the same side of supporting planes). $ \\ $
Let $ x_0 \in D $ such that $ dist(x_0, \partial D) \geq d_0 + 2 \frac{ \sqrt{2n(p+1)} }{(1-p)\hat{c}} $. Since $ dist( \partial D, \partial \Omega ) = d_0 \Rightarrow dist(x_0, \partial \Omega) \geq 2 \frac{ \sqrt{2n(p+1)} }{(1-p)\hat{c}} $ and since $ \Omega $ is convex there exist a ball $ B_R (x_0) \subset \Omega $ for $ R= 2 \frac{ \sqrt{2n(p+1)} }{(1-p)\hat{c}} = 2 \sqrt{n}d(p,\hat{c}) > R_0 = \sqrt{n}d(p,\hat{c}) \;\:, d(p,\hat{c}) $ as defined above. $ \\ $
Therefore we can apply Lemma \ref{LemDC1} in the ball $ B_R (x_0) $ and we have that $ v(x)=0 \:, \forall x \in B_{R_0}(x_0) = \lbrace x \in B_R(x_0) : dist( \partial B_R (x_0),x) \geq R_0 \rbrace \Rightarrow v(x_0)=0 . $

\end{proof}

The results of Lemma \ref{LemDC1} and Lemma \ref{LemDC2} above were proved for the case $ 1< \alpha <2 $, since $ u \in C^{2,\alpha -1} $ by elliptic regularity. However, they also hold for the case where $ 0< \alpha \leq 1 $. The only difference in proving this, is that the differential inequality \eqref{DifIn2} holds weakly and we utilize it together with the weak maximum principle for the comparison argument as in the proof of lemma \ref{LemDC1}. So in order to extend the results of the lemmas above for the case where $ 0< \alpha \leq 1 $, it suffices to prove the following claim. $ \\ $

\begin{lemma}\label{LemWeakDifIn}
\begin{align*}
\Delta v \geq c^2 v^{\frac{\alpha}{2}} \;\;\; \textrm{weakly in} \;\: W^{1,2}(B_R (x_0)).
\end{align*}
\end{lemma}
$ \\ $ 
\begin{proof} $ \\ $

Let $ v \in W^{1,2}(B_R(x_0)) \;\:,\; v $ continuous ($ v = | u-a |^2 $, by Lemma \ref{LemHC}) and $ v \geq 0 . \\ $ We define $ v_{\varepsilon} := \max\lbrace v, \varepsilon \rbrace \;\;,\; 0< \varepsilon < \delta $ (where $ \delta $ as in the above Lemmas). The set $ \lbrace v = \varepsilon \rbrace $ is smooth by Sard's theorem, since $ v $ is smooth away from zero. 

Let $ \phi \in C_0^1(B_R(x_0)) \;\:, \; B_R^{\varepsilon} (x_0) = \lbrace v > \varepsilon \rbrace \cap B_R (x_0) $, we have

\begin{align*}
\begin{gathered}
- \int_{B_R(x_0)} \nabla v \nabla \phi dx = \lim_{\varepsilon \rightarrow 0} \int_{B_R^{\varepsilon}(x_0)} - \nabla v_{\varepsilon} \nabla \phi dx = \liminf_{\varepsilon \rightarrow 0} [ - \int_{B_R^{\varepsilon}(x_0)} \nabla v \nabla \phi dx ] 
\\ \geq \liminf_{\varepsilon \rightarrow 0} [ \int_{B_R^{\varepsilon}(x_0)} \Delta v \phi dx - \int_{ \partial B_R^{\varepsilon}(x_0)} \frac{\partial v}{\partial \nu} \phi dS ] \geq \liminf_{\varepsilon \rightarrow 0} [ \int_{B_R^{\varepsilon}(x_0)} \Delta v \phi dx ] \\ \geq \liminf_{\varepsilon \rightarrow 0} [ \int_{B_R^{\varepsilon}(x_0)} c^2 v^{\frac{\alpha}{2}} \phi dx ] = \lim_{\varepsilon \rightarrow 0} [ \int_{B_R^{\varepsilon}(x_0)} c^2 v^{\frac{\alpha}{2}} \phi dx  ] = \\ \geq \lim_{\varepsilon \rightarrow 0} [ \int_{B_R(x_0)} c^2 v_{\varepsilon}^{\frac{\alpha}{2}} \phi dx - c^2 {\varepsilon}^{\frac{\alpha}{2}} \int_{B_R \setminus B_R^{\varepsilon}} \phi dx ] = \int_{B_R(x_0)} c^2 v^{\frac{\alpha}{2}} \phi dx.
\end{gathered}
\end{align*}

\end{proof}

$ \\ $

\subsection{On the definition of $W^0$}
$ \\ $

In what follows we establish essentially that $ \lim_{\alpha \rightarrow 0} J^{\alpha} = J^0 $ in the $ \Gamma -$ convergence sense. The containment result in Appendix \ref{sec:appendixA} is essential here. 

\begin{align}\label{Ja}
J^\alpha(\Omega,u) = \int_\Omega ( \dfrac{1}{2} |\nabla u|^2 + W^\alpha(u))dx
\end{align}
with 

\begin{align}\label{Wa}
W^\alpha(u):=\displaystyle\prod_{i=1}^N| u-a_i|^\alpha \;\;,\: i \in \lbrace 1,...,N \rbrace \; 0< \alpha < 2. 
\end{align}

We further denote:

\begin{align}\label{W0}
W_0(u):=\chi_{\{u\in S_A\}}
\end{align}
where
$$
A:=\{a_1,\dots,a_N\}
$$ and 

\begin{align}\label{SA}
S_A:=\left\{\sum_{i=1}^N \lambda_i a_i,\textrm{ where }\sum_{i=1}^N \lambda_i=1, \lambda_i\in [0,1),i\in\{1,\dots,N\}\right\}
\end{align}
(i.e. $S_A$ is the convex hull of the points in $A$ except the point themselves). Then 

\begin{align*}
\bar S_A=S_A\cup A
\end{align*}

We have the following:

$ \\ $

\begin{lemma}\label{LemJaLim} Let $(u^{\alpha_k})_{k\in\N}$ be a sequence of functions such that $\alpha_k\to 0$ as $k\to\infty$ and  for any $k\in \mathbb{N} $ the function $u^{\alpha_k}:\R^n\to\R^m$ is an energy minimizer of $J^{\alpha_k}$ as defined in \eqref{Ja}.

We assume that

\begin{align}\label{Lem26eq1}
u^{\alpha_k}(x)\in \bar S_A, \forall x\in \R^n, k\in\N
\end{align}

Then there exists a subsequence relabelled for simplicity as the initial sequence such that:

\begin{align}\label{Lem26eq2}
u^{\alpha_k} \rightharpoonup \tu,\,\textrm{ in } W^{1,2}(\R^n;\R^m), \textrm{ as }k\to\infty
\end{align}
with $\tu$  a local energy minimizer of the functional $J^0$  defined as:
\begin{align}\label{Lem26eq3}
J^0(\Omega, u):=\int_\Omega \dfrac{1}{2} |\nabla u|^2+W^0(u(x))\,dx
\end{align}
(with $W^0$ from \eqref{W0}).
\end{lemma}
\begin{proof} We have 
\begin{align*}
( P) \begin{cases} W^{\alpha_k}(u) \rightarrow W^0(u) \;\;\;\;\textrm{ in }\bar S_A\textrm{ as } k\to\infty\\ 
W^{\alpha_k} \geq 0,\,\forall \alpha_k>0
\end{cases}
\end{align*}

 Arguing along the lines of Lemma \ref{LemJRn-1}, (while taking into account the properties $ ( P) $ and the definition \eqref{Wa} of $W^{\alpha_k}$s) we get:
\begin{align}\label{Lem26eq4}
J^{\alpha_k} (B_r,u^{\alpha_k} ) \leq C r^{n-1}
\end{align}
for all $r\ge 1$,
where  $ C $ depends only on the points $a_1,\dots,a_N$ through the assumed inclusion \eqref{Lem26eq1} (and is independent of $ \alpha_k,k\in\N$).

Out of this uniform bound we claim that there exists $ \tu \in W^{1,2} (\R^n ;\R^m) $ such that:
$ \\ $ (1) $ u^{\alpha_k} \rightharpoonup \tu $ in  $ W^{1,2} (\R^n ;\R^m) $ as $ k \rightarrow \infty$ on a subsequence
$ \\ $ (2) $ \tu $ is a local minimizer of $J^0$.

By the bound \eqref{Lem26eq4} , $ W^{\varepsilon} \geq 0 $ and by the Rellich- Kondrachov theorem, we can obtain, along a subsequence
\begin{align*}
u^{\alpha_k} \rightharpoonup \tu \;\;\;on \;\; W^{1,2} (\R^n; \R^m) 
\end{align*}
and 
\begin{align*}
u^{\alpha_k} \rightarrow \tu \;\;\;on \;\; L^p_{loc} (\R^n ; \R^m) 
\end{align*}
These provide claim (1).

In order to show claim (2) we  note first we have:

\begin{align}\label{Lem26eq5}
J^0(\tu,\Omega)\le \liminf_{\alpha_k\to 0} J^{\alpha_k}(u^{\alpha_k},\Omega)
\end{align}

Indeed, we have by lower semicontinuity 

\begin{align}\label{Lem26eq6}
\int_{\Omega}|\nabla \tu|^2\,dx\le \liminf_{k\to\infty} \int_{\Omega}|\nabla u^{\alpha_k}|^2\,dx
\end{align}

We have that $\tu\in \bar S_A$ and we denote $A_{\tu}:=\{x\in\R^n: \tu (x)\in S_A\}$. Taking into account the specific form \eqref{Wa} of the potential $W^\alpha$ we have, for $\alpha_k\to 0$ as $k\to\infty$:

\begin{align}\label{Lem26eq7}
\int_{A_{\tu}\cap \Omega}\chi_{\{\tu\in S_A\}}\,dx=\int_{A_{\tu}\cap\Omega}dx=\lim_{k\to\infty} \int_{A_{\tu}\cap\Omega} W_{\alpha_k}(u^{ {\alpha}_k} (x))\,dx
\end{align}

Furthermore, since $W^\alpha\ge 0$ we have:
\begin{align}\label{Lem26eq8}
\int_{\Omega\setminus A_{\tu}}\chi_{\{\tu\in S_A\}}\,dx=0\le \lim_{k\to\infty} \int_{\Omega\setminus A_{\tu}} W^{\alpha_k}(u^{ {\alpha}_k} (x))\,dx
\end{align}

The last three estimates provide the claimed relation \eqref{Lem26eq5}. One can then trivially see that:

\begin{align}\label{Lem26eq9}
\inf J^0(\cdot, \Omega) \le J^0(\tu,\Omega)\le \liminf_{\alpha_k\to 0} \inf J^{\alpha_k}(\cdot,\Omega)
\end{align}

We claim now that for an arbitrary $u\in W^{1,2}_{loc}(\R^n;\R^m)$ with $u(x)\in\bar S_A$ for almost all $x\in \R^n$ we have:

\begin{align}\label{Lem26eq10}
\lim_{\alpha_k\to 0}J^{\alpha_k}(u,\Omega)=J^0(u,\Omega)
\end{align}

Indeed we have: 
\begin{align}\label{Lem26eq11}
\int_{A_u\cap \Omega}\chi_{\{u\in S_A\}}\,dx=\int_{A_u\cap\Omega}dx=\lim_{k\to\infty} \int_{A_u\cap\Omega} W^{\alpha_k}(u(x))\,dx
\end{align}
\begin{align}\label{Lem26eq12}
\int_{\Omega\setminus A_u}\chi_{\{u\in S_A\}}\,dx=0= \lim_{k\to\infty} \int_{\Omega\setminus A_u}W^{\alpha_k}(u(x))\,dx
\end{align}
so 

$$\int_{\Omega}|\nabla u|^2+\chi_{\{u\in S_A\}}\,dx=\lim_{k\to\infty} \int_{\Omega}|\nabla u|^2+W^{\alpha_k}(u(x))\,dx,$$ as claimed. 

\bigskip
We note now that \eqref{Lem26eq10} implies:

\begin{align*}
J^0(u,\Omega)=\lim_{\alpha_k\to 0}J^{\alpha_k}(u,\Omega)= \limsup_{\alpha_k\to 0}J^{\alpha_k}(u,\Omega)\ge  \limsup_{\alpha_k\to 0} \inf J^{\alpha_k}(\cdot,\Omega)
\end{align*}
and since this holds for $u$ arbitrary we get:

\begin{align}\label{Lem26eq13}
\inf J^0(\cdot,\Omega)\ge \limsup_{\alpha_k\to 0} \inf J^{\alpha_k}(\cdot,\Omega)
\end{align}

The last inequality, together with \eqref{Lem26eq9} provide the claimed local minimality of $\tu$.

\end{proof}
\textit{Note:} The above Lemma also holds for the class of local minimizers of the energy.

$ \\ \\ $

\section{Proofs}

\subsection{Proof of Proposition \ref{PointwiseEst}}
$ \\ $

\begin{proof}
(i) (cfr \cite{AFS} p.161). Let
\begin{equation}\label{PointwiseEsteq4}
|u(x) -a|<M \;,\;\: || u ||_{C^{\beta}} < \hat{C} =\hat{C}(M) \;,\: x \in \mathcal{O}
\end{equation}
where for the H\"older bound we utilized Lemma \ref{LemHC}. Given $ q \in (0,M) $, assume that
\begin{equation}\label{PointwiseEsteq5}
|u(x_0) -a | \geq q
\end{equation}
Then the H\"older continuity of $ u $ implies that the hypothesis of the Density Estimate \eqref{TheDensityEstimate} is satisfied for
\begin{equation}\label{PointwiseEsteq6}
\lambda = \frac{q}{2} \;\;,\; r_0 = (\frac{q/2}{\hat{C}})^{\frac{1}{\beta}} \;\;,\; \mu_0 = \mathcal{L}^n (B_{r_0}(x_0))
\end{equation}
Therefore
\begin{equation}\label{PointwiseEsteq7}
\mathcal{L}^n (B_r(x_0) \cap \lbrace |u-a| > \frac{q}{2} \rbrace) \geq C r^n \;\;,\; B_r (x_0) \subset \mathcal{O} \;,\: r \geq r_0
\end{equation}
Let
\begin{equation}\label{PointwiseEsteq8}
0< w_{\frac{q}{2}} := \min_{\Sigma} W(z) \;,\: \Sigma = \lbrace |z-a| > \frac{q}{2} \rbrace \cap \lbrace d(z,\lbrace W=0 \rbrace \setminus a) \geq k \rbrace
\end{equation}
From this and the Basic Estimate Lemma \ref{LemBE} we obtain
\begin{equation}\label{PointwiseEsteq9}
w_{\frac{q}{2}} C_1 r^n \leq J_{B_r(x_0)} (u) \leq C_0 r^{n-1}
\end{equation}
which is impossible for
\begin{equation}\label{PointwiseEsteq10}
r > \frac{C_0}{w_{\frac{q}{2}}C_1} 
\end{equation}
Therefore if we set
\begin{equation}
r_q= \frac{2C_0}{w_{\frac{q}{2}}C_1}
\end{equation}
then $ B_{r_q}(x_0) \subset \mathcal{O} $ is incompatible with \eqref{PointwiseEsteq5}.
$\\ $
The proof of (i) is complete. $ \\ $

(ii) Consider the ball $ B_R(x_0) \;,\: R $ to be selected. $ \\ $
Let $ \xi \in B_R(x_0) \\ $

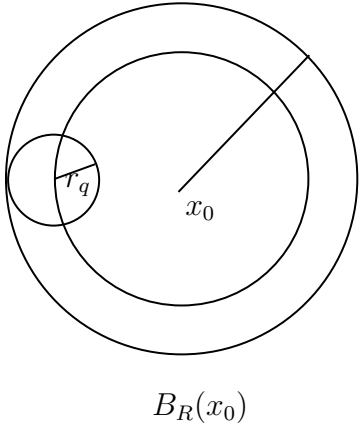
\tikzset{every picture/.style={line width=0.75pt}} 

\begin{tikzpicture}[x=0.75pt,y=0.75pt,yscale=-1,xscale=1]

\draw   (224.8,132.6) .. controls (224.8,83.67) and (264.47,44) .. (313.4,44) .. controls (362.33,44) and (402,83.67) .. (402,132.6) .. controls (402,181.53) and (362.33,221.2) .. (313.4,221.2) .. controls (264.47,221.2) and (224.8,181.53) .. (224.8,132.6) -- cycle ;
\draw   (249.48,132.6) .. controls (249.48,97.3) and (278.1,68.68) .. (313.4,68.68) .. controls (348.7,68.68) and (377.32,97.3) .. (377.32,132.6) .. controls (377.32,167.9) and (348.7,196.52) .. (313.4,196.52) .. controls (278.1,196.52) and (249.48,167.9) .. (249.48,132.6) -- cycle ;
\draw [color={rgb, 255:red, 0; green, 0; blue, 0 }  ,draw opacity=1 ]   (377.8,70.2) -- (311.8,139.2) ;
\draw   (225.88,133.24) .. controls (225.88,120.57) and (236.15,110.29) .. (248.83,110.29) .. controls (261.51,110.29) and (271.79,120.57) .. (271.79,133.24) .. controls (271.79,145.92) and (261.51,156.2) .. (248.83,156.2) .. controls (236.15,156.2) and (225.88,145.92) .. (225.88,133.24) -- cycle ;
\draw    (269.8,125.2) -- (249.48,132.6) ;

\draw (313.8,142.6) node [anchor=north west][inner sep=0.75pt]    {$x_{0}$};
\draw (297,238.4) node [anchor=north west][inner sep=0.75pt]    {$B_{R}( x_{0})$};
\draw (253,128.4) node [anchor=north west][inner sep=0.75pt]    {$r_{q}$};

\end{tikzpicture}

\begin{equation}\label{PointwiseEsteq11}
d( \xi, \partial B_R(x_0)) = r_q \;\:,\; 0< 2q < \rho_0
\end{equation}
where $ r_q $ as in (i) above. Note that by \textbf{(H1)}
\begin{equation}\label{PointwiseEsteq12}
w_{\frac{q}{2}} \geq C^* (\frac{q}{2})^{\alpha} \;,\: r_q = \frac{2C_0}{w_{\frac{q}{2}} C_1} \leq \frac{2C_0}{C_1 C^*}(\frac{q}{2})^{-\alpha}
\end{equation}
and by (i) above
\begin{equation}\label{PointwiseEsteq13}
|u(\xi) - a|<q
\end{equation}
Therefore by \cite{AFS}, Theorem 4.1 originally derived in \cite{AF2}
\begin{equation}
| u(x) -a |<q \;,\; x\in B_{R- r_q} (x_0)
\end{equation}
By \eqref{DifIn2} $ v(x) := |u(x)-a|^2 $ satisfies
\begin{equation}\label{PointwiseEsteq14}
\begin{cases} \Delta v \geq c^2 v^{\frac{\alpha}{2}} \;\: \textrm{weakly in} W^{1,2}(B_{R-r_q}(x_0)) \\ v \leq \delta \;\;\;\;\; \textrm{on} \;\: \partial B_{R-r_q}(x_0)
\end{cases}
\end{equation}\label{PointwiseEsteq15}
and therefore by Lemma \ref{LemDC1}
\begin{equation}
d(y_0, \partial B_{R-r_q}(x_0)) > R_0 \Rightarrow v(y_0) =0 
\end{equation}
where
\begin{equation}\label{PointwiseEsteq16}
R_0 = \frac{\sqrt{n(\alpha +2)}}{(1 - \frac{\alpha}{2})c} q^{1-\frac{\alpha}{2}} \;\;,\; 0 < \alpha<2 \;\:,\; c^2 = 2 \alpha C^*
\end{equation}
Therefore
\begin{equation}\label{PointwiseEsteq17}
u(x) = a \;\;\; \textrm{in} \;\: B_{R- r_q - R_0}(x_0)
\end{equation}
To conclude set $ R=C q^{-\alpha} $ and impose the requirement that
\begin{equation}\label{PointwiseEsteq18}
\frac{C}{2} q^{-\alpha} \leq C q^{-\alpha} - r_q -R_0
\end{equation}
which is satisfied if
\begin{equation}\label{PointwiseEsteq19}
C \geq \frac{2^{\alpha +2} C_0}{C_1 C^*} + 2\frac{\sqrt{n(\alpha +2)}}{(1 - \frac{\alpha}{2})\sqrt{2 \alpha C^*}}(\frac{\rho_0}{2})^{1+ \frac{\alpha}{2}} =: \hat{C}(\alpha , n)
\end{equation}

The proof of Proposition \ref{PointwiseEst} is complete.

\end{proof}

\subsection{Proof of Theorem \ref{ThEquivMin}}
$ \\ $
\begin{proof}
\underline{Step 1} (Existence of a positive minimizer)
$ \\ \\ $

We will be establishing the existence of a map $ u_R \in W^{1,2} (B_R, \mathbb{R}^n) $ that is equivariant, positive and also a minimizer in the equivariant class of
\begin{equation}\label{ThEquivMineq1}
J_{B_R}(u) = \int_{B_R} (\frac{1}{2} |\nabla u|^2 +W(u))dx \;\;,\; B_R = \lbrace |x| < R \rbrace \subset \mathbb{R}^n,
\end{equation}
that satisfies the Basic Estimate
\begin{equation}\label{ThEquivMineq2}
J_{B_r}(u_R) \leq C r^{n-1} \;\;,\; r_0 < r < R \;\:,\; R \geq R_0
\end{equation}
$C$ independent of $ R\:, r . \\ $

We introduce the regularized energy functional

\begin{equation}\label{ThEquivMineq3}
J_{B_R}^{\varepsilon}(u)= \int_{B_R} (\frac{1}{2} |\nabla u|^2 + W^{\varepsilon}(u) )dx
\end{equation}
where $ W^{\varepsilon} $ is obtained from $ W $ by regularizing only at the minima as in Figure below.

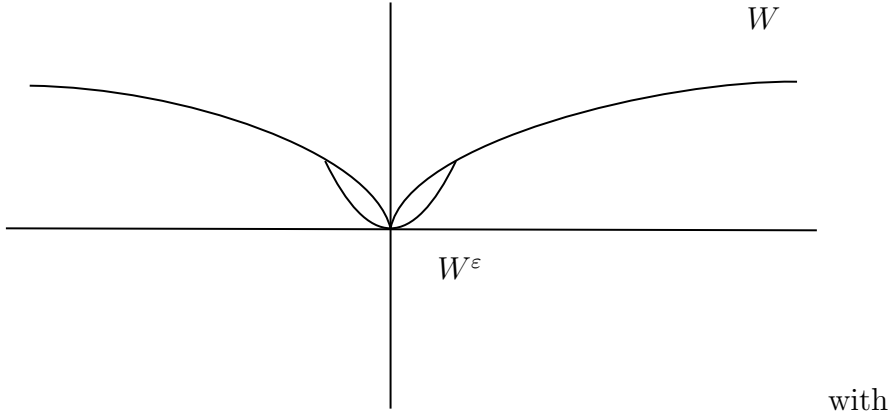
\tikzset{every picture/.style={line width=0.75pt}} 

\begin{tikzpicture}[x=0.75pt,y=0.75pt,yscale=-1,xscale=1]

\draw    (126.8,151.2) -- (535.8,152.2) ;
\draw    (320.8,37.2) -- (320.8,242.2) ;
\draw  [color={rgb, 255:red, 0; green, 0; blue, 0 }  ,draw opacity=1 ] (287.8,117.2) .. controls (309.8,162.53) and (331.8,162.53) .. (353.8,117.2) ;
\draw    (320.8,151.2) .. controls (329.8,107.2) and (458.8,76.2) .. (525.8,77.2) ;
\draw    (138.8,79.2) .. controls (215.8,80.2) and (310.8,112.2) .. (320.8,151.2) ;

\draw (499,38.4) node [anchor=north west][inner sep=0.75pt]    {$W$};
\draw (343,164.4) node [anchor=north west][inner sep=0.75pt]    {$W^{\varepsilon }$};

\end{tikzpicture}
with
\begin{equation}\label{ThEquivMineq4}
(\star) \begin{cases} W^{\varepsilon} \rightarrow W(u) \;,\: \textrm{uniformly on compacts} \\ W^{\varepsilon} \in C^2 \;,\; ||W^{\varepsilon} ||_{C^{\alpha}} < C \;,\; \textrm{for} \;\: W \;\: \textrm{satisfying} \;\: (\textbf{H1}) \:, \\ W^{\varepsilon} \geq 0 \;,\; \lbrace W^{\varepsilon} = 0 \rbrace = \lbrace W = 0 \rbrace \\ W^{\varepsilon}(gu) = W^{\varepsilon}(u) \;\;, \textrm{for all} \;\: g \in G \;\: \textrm{and} \;\: u \in \mathbb{R}^n.
\end{cases}
\end{equation}

We can assume that
\begin{equation}\label{ThEquivMineq5}
W^{\varepsilon}(u) = W(u) \;\;\; \textrm{for} \;\: |u| \geq M >0
\end{equation}
some $ M>0 $, and that the minimizer of $ J_{B_R}^{\varepsilon} $ in the equivariant class satisfies the bound
\begin{equation}\label{ThEquivMineq6}
|u_R^{\varepsilon}| \leq M \;\;\;,\; x \in B_R
\end{equation}
with $ M $ independent of $ \varepsilon $ and $ R $ and that moreover $ u_R^{\varepsilon} $ is positive. Here we are utilizing \cite{AFS} Lemma 6.1.

We begin by establishing the H\"older Estimate \eqref{TheHolderEstimate}, for $ u_R^{\varepsilon} $, with constant $ C $ independent of $ \varepsilon \;, R $. Recall that $ u_R^{\varepsilon} $ is a minimizer in the equivariant class, while \eqref{TheHolderEstimate} was derived under the stronger hypothesis of being a minimizer under arbitrary perturbations. We point out only the necessary modifications of the proof of the Lemma \ref{LemHC}.

We will derive

\begin{equation}\label{ThEquivMinExtraeq1}
| u_R^{\varepsilon} (x) - u_R^{\varepsilon} (y) | \leq C | x-y | \: \textrm{ln} |x-y|^{-1} \;\;\;, \forall \: x,y \in B_R(0) \setminus B_1(0)
\end{equation}
with $ |x-y| \leq \frac{1}{2}\;\:,\; R \geq 2 . $

Notice that we can cover $ F_R \cap (B_R(0) \setminus B_1(0)) = : F_{R,D} $ where $ F_R = F \cap B_R(0) $ by two types of balls $ B_{\frac{1}{4}} (x_0): \\ $ (a) Balls entirely contained in $ F_{R,D} \;,\; B_{\frac{1}{4}} \subset F_{R,D} \: , \\ $ (b) balls $ B_{\frac{1}{4}}(x_0) $ having their center in the wall of $ F_R $ which is made up of reflection planes in $ G_a . \\ $ Notice that both types can be equivariantly extended over $ B_R(0) \setminus B_1(0) $ as sets.

Fix now $ B_r(x_0) \;,\: r < \frac{1}{4} $ as in the proof of \eqref{Lem21eq1}. Due to the equivariant extension of $ v_r $ there, and the minimality of $ u_R^{\varepsilon} $ in the equivariant class, we see that $ u_R^{\varepsilon} $ has the minimizing property on $ B_r(x_0) $ and so \eqref{Lem21eq3} applies as before. The rest of the argument is unchanged.

Thus \eqref{ThEquivMinExtraeq1} is established. $ \\ $

\tikzset{every picture/.style={line width=0.75pt}} 

\begin{tikzpicture}[x=0.75pt,y=0.75pt,yscale=-1,xscale=1]

\draw   (121.1,161.6) .. controls (121.1,143.21) and (136.01,128.3) .. (154.4,128.3) .. controls (172.79,128.3) and (187.7,143.21) .. (187.7,161.6) .. controls (187.7,179.99) and (172.79,194.9) .. (154.4,194.9) .. controls (136.01,194.9) and (121.1,179.99) .. (121.1,161.6) -- cycle ;
\draw    (437.8,161.2) -- (154.4,161.6) ;
\draw    (385.8,7.6) -- (154.4,161.6) ;
\draw    (385.8,7.6) .. controls (410.8,14.6) and (449.8,128.2) .. (437.8,161.2) ;
\draw   (258.8,161.6) .. controls (258.8,156.35) and (263.05,152.1) .. (268.3,152.1) .. controls (273.55,152.1) and (277.8,156.35) .. (277.8,161.6) .. controls (277.8,166.85) and (273.55,171.1) .. (268.3,171.1) .. controls (263.05,171.1) and (258.8,166.85) .. (258.8,161.6) -- cycle ;
\draw   (311.8,127.6) .. controls (311.8,122.08) and (316.28,117.6) .. (321.8,117.6) .. controls (327.32,117.6) and (331.8,122.08) .. (331.8,127.6) .. controls (331.8,133.12) and (327.32,137.6) .. (321.8,137.6) .. controls (316.28,137.6) and (311.8,133.12) .. (311.8,127.6) -- cycle ;

\draw (78,143.4) node [anchor=north west][inner sep=0.75pt]    {$B_{1}( 0)$};
\draw (439,29.4) node [anchor=north west][inner sep=0.75pt]    {$F_{R}$};
\draw (62,220.4) node [anchor=north west][inner sep=0.75pt]    {$Fig:$};
\draw (97,221) node [anchor=north west][inner sep=0.75pt]   [align=left] {Typical};
\draw (150,219.4) node [anchor=north west][inner sep=0.75pt]    {$B_{\frac{1}{4}}( x_{0}) \ 's$};
\draw (219,220) node [anchor=north west][inner sep=0.75pt]   [align=left] {covering the fundamental region};
\draw (99,240) node [anchor=north west][inner sep=0.75pt]   [align=left] {and extensible equivariantly on};
\draw (315,238.4) node [anchor=north west][inner sep=0.75pt]    {$B_{R}( 0) \setminus B_{1}( 0) .$};

\end{tikzpicture}

$ \\ $

Now we will proceed to establish \eqref{ThEquivMineq2},
\begin{equation}\label{ThEquivMinExtraeq2}
J_{B_r(0)}(u_R^{\varepsilon}) \leq Cr^{n-1} \;\;\:,\; \forall \: r \in (2,R-1)
\end{equation}
with $ C $ constant independent of $ \varepsilon $ and $ R \:, \; C=C(M) . \\ $
We follow \cite{AFS} Proposition 6.1, and for $ 2<r< R-1 $ we define
\begin{equation}\label{ThEquivMinExtraeq3}
u_{aff}(x) = \begin{cases} d(x, \partial D) a_1 \;\;\; , \textrm{for} \;\: x \in D_R \;\: \textrm{and} \;\: d(x, \partial D) \leq 1 \\
a_1 \;\;\;\;\;\;\;\;\;\;\;\;\;\;\;\;, \textrm{for} \;\: x \in D_R \;\: \textrm{and} \;\: d(x,\partial D) \geq 1
\end{cases}
\end{equation}
where $ D_R = D \cap B_R $ and extend equivariantly in $ B_R $. Since $ u_{aff} $ vanishes on $ \partial D $, the extended map is also continuous. As it is well known, the distance is 1-Lipschitz and therefore in $ W^{1,\infty} (B_R). $ Fix now a number $ h \in (0,1) $ and for $ r \in (2,R-1) $ define
\begin{equation}\label{ThEquivMinExtraeq4}
\hat{u}_R^{\varepsilon}(x) = \varphi(1- \frac{|x|-(r-h)}{h}) u_{aff}(x) + \phi(\frac{|x|-(r-h)}{h}) u_R(x)
\end{equation}
where $ \phi : \mathbb{R} \rightarrow [0,1] $ is a fixed $ C^1 $ function such that $ \phi (s) =0 $, for $ s \leq 0 $ and $ \phi (s) =1 $, for $ s \geq 1 $. Note that $ \hat{u}_R^{\varepsilon} \in W^{1,2}_E (B_R(0); \mathbb{R}^n) $ (equivariant), and most importantly $ \hat{u}_R^{\varepsilon}= u_R^{\varepsilon} $ on $ \partial B_r(0). $ Moreover $ \hat{u}_R = u_{aff} $ in $ B_{r-h}(0) $ and $ \hat{u}_R^{\varepsilon} = u_R^{\varepsilon} $ on $ B_R(0) \setminus B_1(0) $ and $ u_{aff}=a_1 $ if $ d(x, \partial D) \geq 1 $. By the minimality of $ u_R^{\varepsilon} $ we have
\begin{equation}\label{ThEquivMinExtraeq5}
\begin{gathered}
J_{B_r(0)}(u_R^{\varepsilon}) \leq J_{B_r(0)}(\hat{u}_R^{\varepsilon})  \\ = \int_{B_{r-h} \cap \lbrace d(x,\partial D) \leq 1 \rbrace} ( \frac{1}{2}| \nabla \hat{u}_R^{\varepsilon}|^2 + W( \hat{u}_R^{\varepsilon}) ) dx + \int_{B_r \setminus B_{r-h}} (\frac{1}{2} | \nabla \hat{u}_R^{\varepsilon} |^2 + W(\hat{u}_R^{\varepsilon}) )dx \\ \leq C_1 (r-h)^{n-1} + C_2 r^{n-1}
\end{gathered}
\end{equation}
where for the estimate of the $ 2^{nd} $ term we used the H\"older estimate above and the analogous \eqref{LemJRn-1eq2}, \eqref{LemJRn-1eq3}.

Hence \eqref{ThEquivMinExtraeq2} is established.

$ \\ $
Thus for any $ R>0 $ there exists $ C_R >0 $, independent of $ \varepsilon >0 $, such that
\begin{equation}\label{ThEquivMineq8}
\int_{B_R} (\frac{1}{2} |\nabla u_R^{\varepsilon}|^2 + W^{\varepsilon}(u_R^{\varepsilon}))dx < C_R
\end{equation}
Out of the above uniform bounds we claim that there exists $ u_R \in W^{1,2} (B_R ; {\mathbb{R}}^m) $ such that

$ \\ $ (1) $ u_R^{\varepsilon} \rightharpoonup u_R $ weakly in $ W^{1,2} (B_R; {\mathbb{R}}^m) $ as $ \varepsilon \rightarrow 0 $ on a subsequence,

$ \\ $ (2) $ u_R $ is a minimizer of 
\begin{align*}
J_{B_R}(u) = \int_{B_R} ( \frac{1}{2} | \nabla u |^2 + W(u)) dx \;\: ,
\end{align*}

$ \\ $ (3) $ J_{B_r} (u_R ) \leq C r^{n-1} $ with $ C $ independent of $ \varepsilon $ and $ R \;\: $,

$ \\ $ (4) $ u_R $ is equivariant and positive.

$ \\ $

By \eqref{ThEquivMineq8} and $ W^{\varepsilon} \geq 0 $ and the Rellich-Kondrachov theorem, we can obtain, for a subsequence
\begin{align*}
u_R^{\varepsilon} \rightharpoonup u_R \;\;\; \textrm{on} \;\; W^{1,2} (B_R ; {\mathbb{R}}^m) 
\end{align*}
and 
\begin{align*}
u_R^{\varepsilon} \rightarrow u_R \;\;\;on \;\; L^p (B_R ; {\mathbb{R}}^m) 
\end{align*}
These establish claims (1) and (4).

In order to show claim (2) we take $ \phi \in C^{\infty}_c ({\mathbb{R}}^n) \;\:, \; supp \phi \subset K \subset B_R $. Then by minimality we have:
\begin{align*}
\begin{gathered} J_{B_R}^{\varepsilon} ( u^{\varepsilon}_R + \phi) -  J_{B_R}^{\varepsilon} ( u^{\varepsilon}_R ) \geq 0 \\
\Leftrightarrow \int_{B_R} ( \nabla u^{\varepsilon}_R \nabla \phi + \frac{1}{2} | \nabla \phi |^2 + W^{\varepsilon} ( u^{\varepsilon}_R + \phi ) - W^{\varepsilon} ( u^{\varepsilon}_R) )dx \geq 0
\end{gathered}
\end{align*}
Let $ I_1^{\varepsilon} := \int_{B_R} \nabla u^{\varepsilon}_R \nabla \phi dx $ and $ I_2^{\varepsilon} := \int_{B_R} ( W^{\varepsilon} ( u^{\varepsilon}_R + \phi ) - W^{\varepsilon} ( u^{\varepsilon}_R))dx $.

Thanks to (1) before we have $ I_1^{\varepsilon} \rightarrow I_1 = \int_{B_R} \nabla u_R \nabla \phi dx \\ $ we split:
\begin{align*}
I_2 = \int_{B_R} (W^{ \varepsilon} (u^{\varepsilon}_R + \phi) - W (u^{\varepsilon}_R + \phi))dx + \int_{B_R} (W( u^{\varepsilon}_R + \phi) - W^{\varepsilon} (u^{\varepsilon}_R))dx
\end{align*}
Let $ I^{\varepsilon}_{21} := \int_{B_R} (W^{ \varepsilon} (u^{\varepsilon}_R + \phi) - W (u^{\varepsilon}_R + \phi))dx $ and $ I^{\varepsilon}_{22} := \int_{B_R} (W( u^{\varepsilon}_R + \phi) - W^{\varepsilon} (u^{\varepsilon}_R ))dx \;\;,\; I^{\varepsilon}_{21} \rightarrow 0 $ as $ \varepsilon \rightarrow 0 $ because of the uniform bound $ | u^{\varepsilon}_R (x) | \leq M $ the uniform convergence on compacts of $ W^{\varepsilon} $ to $ W $ and the dominated convergence theorem. $ \\ $
Also $ I^{\varepsilon}_{22} \rightarrow I_{22} = \int_{B_R} ( W( u_R + \phi) - W(u_R))dx $ because of the $ L^p $ convergence of $ u^{\varepsilon}_R $ to $ u_R $, dominated convergence and continuity of $ W . \\ $

Thus we establish the claimed relation (2). $ \\ $
In order to get the claimed relation (3) we recall 
\begin{align*}
J^{\varepsilon}_{B_r} ( u^{\varepsilon}_R ) = \int_{B_r} (\frac{1}{2} |\nabla u^{\varepsilon}_R |^2 + W^{\varepsilon} ( u^{\varepsilon}_R))dx \leq C r^{n-1}
\end{align*}
with $ C $ depending only on $ M $, but not on $ R $ nor on $ \varepsilon $. 

As $ u_R^{\varepsilon} \rightharpoonup u_R $ in $ W^{1,2} \Rightarrow \int_{B_R} | \nabla u_R |^2 dx \leq \liminf \int_{B_R} \frac{1}{2} |\nabla u^{\varepsilon}_R|^2 dx $ and we have
\begin{align*}
\int_{B_R} W^{\varepsilon} ( u^{\varepsilon}_R) dx \rightarrow \int_{B_R} W(u_R) dx
\end{align*}
arguing as in the treatment of the $ I_2 $ before. $ \\ $

\textbf{Claim:} There exists $ \overline{u} \in W^{1,2}_{loc} ( {\mathbb{R}}^n ;{\mathbb{R}}^m) $ nontrivial equivariant, positive and minimizer of 
\begin{equation}\label{ThEquivMineq9}
J_{\Omega}(u) = \int_{\Omega} (\frac{1}{2} | \nabla u |^2 + W(u))dx
\end{equation}
In addition, $ \overline{u} $ satisfies the estimate
\begin{equation}\label{ThEquivMineq10}
J_{B_r} ( \overline{u}) \leq c r^{n-1}
\end{equation}
$ \\ $
\textit{Proof.}

We have that out of the uniform bound $ J_{B_r} ( u_R) \leq c r^{n-1} $, we get as before, in the proof of the claims (1)-(4) that $ u_R \rightharpoonup \overline{u} $ in $ W^{1,2}_{loc} ( {\mathbb{R}}^n ; {\mathbb{R}}^m) $ and that $ \overline{u} $ is equivariant and positive. We can argue similarly as in the proof of (2) above to get that $ \overline{u} $ is a minimizer of $ J_{\Omega} $ defined in \eqref{ThEquivMineq9}, \eqref{ThEquivMineq10} follows from (3).
$ \\ \\ $



$ \\ \\ $
\textit{Step 2.} (Existence of a free boundary) $ \\ $

We utilize that $ D $ contains a unique zero $ a_1 $ of $ W $ and that by equivariance we can restrict $ u $ in $ D $ and note that
\begin{align*}
d(u(D), \lbrace W = 0 \rbrace \setminus a_1) \geq k >0
\end{align*}

For implementing Proposition \ref{PointwiseEst} we need a couple of observarions. Firstly $ u $ is minimizing in the class of equivariant positive maps. We recall that in the proof of Proposition \ref{PointwiseEst} the density estimate \eqref{TheDensityEstimate} is utilized. We note that in the proof of the density estimate the energy comparison maps are obtained by reducing the modulus of the map $ q^u (x) = |u(x) -a_1| $ and leaving the angular part $ \nu^u (x) $ unchanged, $ u(x) = a_1 + q^u(x) \nu^u(x) \;\:,\; \sigma (x) = a_1 + q^{\sigma}(x) \nu^u(x) \;\:,\; 0 \leq q^{\sigma}(x) \leq q^u(x) . \\ 
\\ $

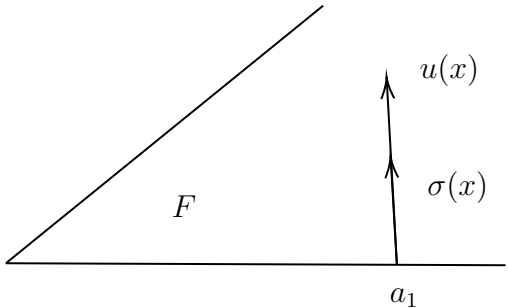
\tikzset{every picture/.style={line width=0.75pt}} 

\begin{tikzpicture}[x=0.75pt,y=0.75pt,yscale=-1,xscale=1]

\draw    (161.8,199.2) -- (413.8,200.2) ;
\draw    (161.8,199.2) -- (321.8,69.2) ;
\draw    (359,200) -- (355.92,147.2) ;
\draw [shift={(355.8,145.2)}, rotate = 446.66] [color={rgb, 255:red, 0; green, 0; blue, 0 }  ][line width=0.75]    (10.93,-3.29) .. controls (6.95,-1.4) and (3.31,-0.3) .. (0,0) .. controls (3.31,0.3) and (6.95,1.4) .. (10.93,3.29)   ;
\draw    (359,200) -- (353.91,107.2) ;
\draw [shift={(353.8,105.2)}, rotate = 446.86] [color={rgb, 255:red, 0; green, 0; blue, 0 }  ][line width=0.75]    (10.93,-3.29) .. controls (6.95,-1.4) and (3.31,-0.3) .. (0,0) .. controls (3.31,0.3) and (6.95,1.4) .. (10.93,3.29)   ;

\draw (244,163.4) node [anchor=north west][inner sep=0.75pt]    {$F$};
\draw (369,92.4) node [anchor=north west][inner sep=0.75pt]    {$u( x)$};
\draw (373,151.4) node [anchor=north west][inner sep=0.75pt]    {$\sigma ( x)$};
\draw (354,210.4) node [anchor=north west][inner sep=0.75pt]    {$a_{1}$};

\end{tikzpicture}

$ \\ $

Therefore by the convexity of $ F $ the comparison map $ \sigma (x) $ is also positive, $ \sigma ( \overline{F} ) \subset \overline{F} $, and it can be extended equivariantly from $ F $ to $ \mathbb{R}^n $ since $ B_R(x_0) \subset F $ or $ B_R(x_0) \subset D $ with $ x_0 \in \partial F $, in the boundary of $ F $, which consists of reflection planes in $G_{a_1}. $

Thus Proposition \ref{PointwiseEst} (ii) can be applied for a fixed $ q $, with $ 2 q \leq \rho $, to produce the estimate
\begin{equation}\label{ThEquivMineq11}
B_{C q^{-\alpha}} (x_0) \subset D \Rightarrow u(x) \equiv a_1 \;\: \textrm{in} \; B_{\frac{C}{2}q^{-\alpha}}(x_0)
\end{equation}
for $ C \geq \hat{C}(\alpha,n). \\ $

By taking a sequence of $ C's $ tending to infinity via a covering argument we see that
\begin{equation}\label{ThEquivMineq12}
u(x) \equiv a_1 \;\; \textrm{if} \; d(x, \partial D) \geq \hat{C}(\alpha,n) q^{-\alpha}
\end{equation}

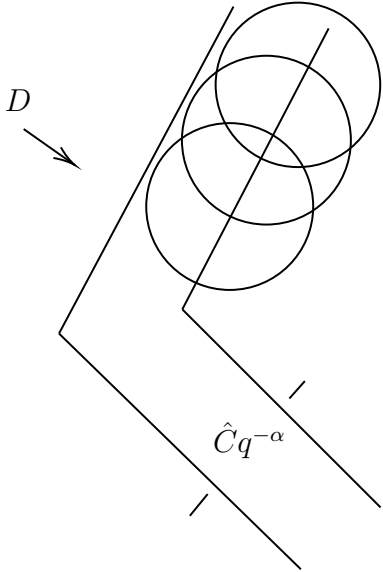
\tikzset{every picture/.style={line width=0.75pt}} 

\begin{tikzpicture}[x=0.75pt,y=0.75pt,yscale=-1,xscale=1]

\draw    (198.8,182.2) -- (320.8,301.2) ;
\draw    (286.8,17.2) -- (198.8,182.2) ;
\draw    (261,170) -- (361,270) ;
\draw    (334.8,28.2) -- (261,170) ;
\draw    (181,79) -- (205.17,96.05) ;
\draw [shift={(206.8,97.2)}, rotate = 215.2] [color={rgb, 255:red, 0; green, 0; blue, 0 }  ][line width=0.75]    (10.93,-3.29) .. controls (6.95,-1.4) and (3.31,-0.3) .. (0,0) .. controls (3.31,0.3) and (6.95,1.4) .. (10.93,3.29)   ;
\draw   (277.8,56.7) .. controls (277.8,33.78) and (296.38,15.2) .. (319.3,15.2) .. controls (342.22,15.2) and (360.8,33.78) .. (360.8,56.7) .. controls (360.8,79.62) and (342.22,98.2) .. (319.3,98.2) .. controls (296.38,98.2) and (277.8,79.62) .. (277.8,56.7) -- cycle ;
\draw   (260.8,84.6) .. controls (260.8,61.07) and (279.87,42) .. (303.4,42) .. controls (326.93,42) and (346,61.07) .. (346,84.6) .. controls (346,108.13) and (326.93,127.2) .. (303.4,127.2) .. controls (279.87,127.2) and (260.8,108.13) .. (260.8,84.6) -- cycle ;
\draw   (242.8,118.1) .. controls (242.8,94.85) and (261.65,76) .. (284.9,76) .. controls (308.15,76) and (327,94.85) .. (327,118.1) .. controls (327,141.35) and (308.15,160.2) .. (284.9,160.2) .. controls (261.65,160.2) and (242.8,141.35) .. (242.8,118.1) -- cycle ;
\draw    (264.8,274.2) -- (273.8,263.2) ;
\draw    (322.8,206.2) -- (315,215) ;

\draw (170,57.4) node [anchor=north west][inner sep=0.75pt]    {$D$};
\draw (275,225.4) node [anchor=north west][inner sep=0.75pt]    {$\hat{C} q^{-\alpha }$};

\end{tikzpicture}



$ \\ $

The proof of Theorem \ref{ThEquivMin} is complete.

\end{proof}

\subsection{Proof of Proposition \ref{PropLowerBd}}

$ \\ $
\begin{proof}
From the assumption \eqref{PropLowerBdeq2} and the Basic Estimate we have
\begin{align*}
\int_{B_R(x_0)} \chi_{\lbrace u \neq a_i \rbrace}dx = \int_{B_R(x_0)} \chi_{A^c}(u)dx \leq CR^{n-1}
\end{align*}
But
\begin{align*}
\int_{B_R(x_0)} \chi_{\lbrace u \neq a_i \rbrace}dx = \mathcal{L}^n ( \lbrace |u - a_i|>0 \rbrace \cap B_R(x_0))
\end{align*}
Hence
\begin{align*}
\mathcal{L}^n ( \lbrace u = a_i \rbrace \cap B_R(x_0)) \geq | B_R(x_0)| - cR^{n-1} \geq CR^n \;\; , \; R> R_0.
\end{align*}
\end{proof}

\subsection{Proof of Proposition \ref{PropFinitePer}}
$ \\ $

\begin{proof}
Let
\begin{align*}
0 < \theta < d_0 := \min \lbrace | a_i - a_j | : i \neq j \;,\: i,j \in \lbrace 1,...,N \rbrace \rbrace
\end{align*}
$ \theta $ arbitrary otherwise. $ \\ \\ $
1. We claim that there exist at least two distinct points $ a_i \neq a_j $ in $ A $ such that
\begin{align*}
\mathcal{L}^n (B_R(x_0) \cap \lbrace | u-a_k| \leq \theta \rbrace ) \geq C_k R^n \;\;\;,\; R \geq R_0 \;\:,\; k =i,j
\end{align*}
$ C_k = C_k ( \theta ) . \\ \\ $
\textit{Proof of the Claim.}
Since $ u $ is a nonconstant minimizer, there is $ x_1 $ such that $ u(x_1) \neq a_1 $
\begin{align*}
\Rightarrow \mathcal{L}^n (B_{\tilde{R}_0}(x_1) \cap \lbrace | u-a_1 | > \lambda \rbrace) \geq \mu_0 \;\;\; (\textrm{by continuity, for some} \;\: \tilde{R}_0 \:, \mu_0 >0 \;\: \textrm{and} \;\: \lambda >0 \;\: \textrm{small})
\end{align*}
and therefore by the Density Estimate \eqref{TheDensityEstimate} we have:
\begin{equation}\label{PropFinitePereq4}
\mathcal{L}^n (B_R(x_1) \cap \lbrace | u-a_1 | > \lambda \rbrace) \geq cR^n \;\;,\; R \geq \tilde{R}_0 .
\end{equation}
Notice that by \eqref{PropFinitePereq4}, there is $ R_1(x_0) >0 $ such that
\begin{equation}\label{PropFinitePereq5}
\mathcal{L}^n (B_R(x_0) \cap \lbrace | u-a_1 | > \lambda \rbrace) \geq c_1R^n \;\;,\; R \geq R_1(x_0) .
\end{equation}
Similarly, since $ u \neq a_k $ there is $ x_k $ such that $ u(x_k) \neq a_k $ and we can repeat the arguments above with $ x_k $ in the place of $ x_1 $ to obtain
\begin{equation}\label{PropFinitePereq6}
\mathcal{L}^n (B_R(x_0) \cap \lbrace | u-a_k | > \lambda \rbrace ) \geq c_k R^n \;\;,\; R \geq R_k \;,\: k=2,...,N
\end{equation}
for some small $ \lambda >0 . \\ $
By Remark 5.4 in \cite{AFS}, $ \forall \; \lambda_1,...,\lambda_N \in (0,d_0) $ we have
\begin{equation}\label{PropFinitePereq7}
\mathcal{L}^n (B_R(x_0) \cap \lbrace | u-a_k | > \lambda_k \rbrace) \geq c_k R^n \;\;,\; R \geq R_0 \;,\: (R_0 = \max_{k \in \lbrace 1,...,N \rbrace} R_k).
\end{equation}
So, if $ \lambda < d_0 - \theta $ and $ | u-a_1| \leq \theta < d_0 \leq | a_1 - a_2 | $
\begin{align*}
\Rightarrow | u-a_2 | \geq | a_1 -a_2 | - \theta > \lambda >0 \Rightarrow \lbrace |u-a_1| \leq \theta \rbrace \subset \lbrace |u-a_2| > \lambda \rbrace.
\end{align*}
Thus
\begin{align}\label{PropFinitePerDefA2}
A_2 := \bigcup_{k=1 \;,\; k\neq 2}^N \lbrace |u-a_k| \leq \theta \rbrace \subset \lbrace |u-a_2| > \lambda \rbrace
\end{align}
\begin{equation}\label{PropFinitePereq8}
\begin{gathered}
\Rightarrow A_2 \cup [ \lbrace |u-a_2| > \lambda \rbrace \cap A_2^c ] = \lbrace |u-a_2| > \lambda \rbrace \\
\Leftrightarrow A_2 \cup [ \lbrace |u-a_2| > \lambda \rbrace \cap ( \bigcap_{k=1 \;,\; k \neq 2}^N \lbrace |u-a_k| > \theta \rbrace ) ] = \lbrace |u-a_2| > \lambda \rbrace
\end{gathered}
\end{equation}
and from the Basic Estimate \eqref{TheBasicEstimate} and the hypothesis \textbf{(H1)} on $ W $ we have
\begin{align*}
\mathcal{L}^n(B_R(x_0) \cap \lbrace |u-a_2| > \lambda \rbrace \cap ( \bigcap_{k=1 \;,\; k \neq 2}^N \lbrace |u-a_k| > \theta \rbrace )) \leq \overline{c}R^{n-1}
\end{align*}
Hence, by \eqref{PropFinitePereq7} and \eqref{PropFinitePereq8} it holds
\begin{align*}
\mathcal{L}^n(B_R(x_0) \cap A_2) \geq \overline{c}_2 R^n \Leftrightarrow \mathcal{L}^n (B_R(x_0) \cap ( \bigcup_{k=1 \;,\; k\neq 2}^N \lbrace |u-a_k| \leq \theta \rbrace) ) \geq \overline{c}_2 R^n
\end{align*}
and similarly, if $ A_l := \bigcup_{k=1 \;,\; k\neq l}^N \lbrace |u-a_k| \leq \theta \rbrace \;\:,\; l=1,2,...,N \: $, we have
\begin{align*}
\mathcal{L}^n (B_R(x_0) \cap (\bigcup_{k\neq l} \lbrace |u-a_k| \leq \theta \rbrace) ) \geq \overline{c}_l R^n \;\;,\; R \geq R_0 
\end{align*}
for all $ l=1,2,...,N . \\ $
Therefore there exist at least two $ i,j \: \in \lbrace 1,...,N \rbrace $ such that
\begin{align*}
\mathcal{L}^n (B_R(x_0) \cap \lbrace |u-a_k| \leq \theta \rbrace ) \geq \overline{c}_k R^n \;\; ,\: R \geq R_0 \;\:,\: k=i,j,
\end{align*}
and the claim is proved.
\begin{flushright}
$ \square $
\end{flushright}
$ \\ $
2. We now proceed to conclude the proof of Proposition \ref{PropFinitePer}. 
$ \\ $
Let $ \; \mathcal{A}_k^R := \overline{B_R(x_0)} \cap \lbrace |u-a_k| \leq \theta \rbrace \;\;,\; k=i,j $
\begin{equation}\label{PropFinitePereq9}
\begin{gathered}
\int_{\mathcal{A}_i^R} \chi_{\lbrace u \neq a_i \rbrace}(x) dx = \mathcal{L}^n ( \lbrace | u-a_i |>0 \rbrace \cap \mathcal{A}_i^R ) \\ = \mathcal{L}^n ( \bigcap_{k=1}^N \lbrace |u-a_k|>0 \rbrace \cap \mathcal{A}_i^R) \;\;\;(\textrm{by} \;\: \eqref{PropFinitePerDefA2}) \\ = \int_{\mathcal{A}_i^R} W^0 (u)dx \leq cR^{n-1} \;\;\;(\textrm{by the Basic Estimate} \;\: \eqref{TheBasicEstimate})
\end{gathered}
\end{equation}
\begin{equation}\label{PropFinitePereq10}
\begin{gathered}
\mathcal{L}^n (\lbrace u=a_i \rbrace \cap \mathcal{A}_i^R) = \mathcal{L}^n(\mathcal{A}_i^R) - \mathcal{L}^n (\lbrace u \neq a_i \rbrace \cap \mathcal{A}_i^R) \\
\geq c_i R^n - \mathcal{L}^n (\lbrace u \neq a_i \rbrace \cap \mathcal{A}_i^R) \;\;\; (\textrm{by Step 1.}) \\ \geq c_i R^n - cR^{n-1} \geq C_i R^n \;\;,\; R \geq R_0 \;\; (\textrm{by} \;\: \eqref{PropFinitePereq9})
\end{gathered}
\end{equation}
Similarly for $ \lbrace u=a_j \rbrace . \\ $ 

Now, for obtaining \eqref{PropFinitePereq3}, we utilize the isoperimetric inequality (see for example \cite{EG})

\begin{equation}\label{PropFinitePereq10}
\min \lbrace \mathcal{L}^n(\overline{B_R(x_0)} \cap E_i) \:,\: \mathcal{L}^n ( \overline{B_R(x_0)} \setminus E_i) \rbrace^{1-\frac{1}{n}} \leq 2 \hat{c} \: || \partial E_i ||(B_R(x_0))
\end{equation}
with $ E_i = \lbrace u(x) =a_i \rbrace \;\: (E_j= \lbrace u(x) =a_j \rbrace) .$ Utilizing \eqref{PropFinitePereq2}, we have
\begin{align*}
\mathcal{L}^n (\overline{B_R(x_0)} \cap E_i) \geq c_i R^n
\end{align*}
On the other hand
\begin{align*}
\overline{B_R(x_0)} \setminus E_i \supset \overline{B_R(x_0)} \cap E_j
\end{align*}
and once more by \eqref{PropFinitePereq2}
\begin{align*}
\mathcal{L}^n (\overline{B_R(x_0)} \cap E_j) \geq c_j R^n
\end{align*}
Thus the lower bound \eqref{PropFinitePereq3} follows.

The proof of Proposition \ref{PropFinitePer} is complete.
\end{proof}
$ \\ $

\subsection{Proof of Proposition \ref{PropSymmetrica=0}}
$ \\ \\ $
\begin{proof}
1. Here we require $ N=m+1 $ and invoke Lemma \ref{LemJaLim}, and thus produce an equivariant, positive minimizer for $ \alpha =0 $ satisfying the Basic Estimate \eqref{ThEquivMineq10}. We note that from equivariance and \eqref{ThEquivMineq10} it follows that $ u \neq $ constant (if $ u \equiv $ constant, from equivariance we would have that $ u \equiv (0,...,0) $ which contradicts the Basic Estimate \eqref{ThEquivMineq10} since $ (0,...,0) \notin \lbrace W=0 \rbrace $).
$ \\ \\ $ 2. By Proposition \ref{PropFinitePer} we have that there exist $ R_0 >0 $ and at least two distinct $ a_i \neq a_j \;(i,j \; \in \lbrace 1,...,N+1 \rbrace )$ such that
\begin{equation}\label{PropSymmetrica=0eq1}
\mathcal{L}^n ( B_R(0) \cap \lbrace u =a_k \rbrace) \geq c_k R^n \;\;,\; R\geq R_0 \;,\; k =i,j.
\end{equation}
We partition $ \mathbb{R}^n $ in $ D^1,...,D^{N+1} $ (see \textbf{(H3)}) where in each $ D^i $ there is a unique global minimum of $ W $ (i.e. $ a_i $ , and $ D^1 $ is denoted as $ D $). Thus $ u \neq a_j $ in the region $ D^i \; (i \neq j )$, so from \eqref{PropSymmetrica=0eq1} we have
\begin{equation}\label{PropSymmetrica=0eq2}
\mathcal{L}^n(B_R(0) \cap \lbrace u = a_i \rbrace) = \mathcal{L}^n ( D_R^i \cap \lbrace u =a_i \rbrace) \geq c_i R^n \;\;,\; R \geq R_0 \;,\; D_R^i =D^i \cap B_R(0)
\end{equation}
and from the equivariance of $ u $ we obtain
\begin{equation}\label{PropSymmetrica=0eq3}
\mathcal{L}^n( D_R^k \cap \lbrace u =a_k \rbrace ) \geq c_k R^n \;\;,\; R \geq R_0 \;,\; k=1,...,N+1.
\end{equation}
$ \\ $ 3. Finally, from the Basic Estimate \eqref{ThEquivMineq10}, we have
\begin{equation}\label{PropSymmetrica=0eq4}
\mathcal{L}^n(B_R(0) \cap (\bigcap_{i=1}^{N+1} \lbrace u \neq a_i \rbrace) = \int_{B_R(0)} W^0(u) dx \leq C R^{n-1}
\end{equation}
and therefore
\begin{equation}\label{PropSymmetrica=0eq5}
\mathcal{L}^n (D_R^1 \cap \lbrace u \neq a_1 \rbrace ) = \mathcal{L}^n (D_R^1 \cap (\bigcap_{i=1}^{N+1} \lbrace u \neq a_i \rbrace) \leq CR^{n-1}.
\end{equation}

The proof of Proposition \ref{PropSymmetrica=0} is complete.
\end{proof}

$ \\ \\ \\ $

\begin{appendix}
\begin{LARGE}
\textbf{Appendix}
\end{LARGE}
\section{The Containment}
\label{sec:appendixA}
$ \\ $

The following result was established by the first author and P. Smyrnelis in unpublished work \cite{AS}. We reproduce it here for the convenience of the reader. For related applications of the method of proof we refer to \cite{Sm}. $ \\ $

\begin{proposition}\label{PrCont} (\cite{AS}) $ \\ $
Let $ u : {\mathbb{R}}^n \rightarrow {\mathbb{R}}^m $ be a bounded ($ | u(x) | < M $) critical point of the functional
\begin{align*}
J (u) = \int ( \frac{1}{2} | \nabla u|^2 + W(u)) dx
\end{align*}
in the sense that $ \forall \Omega \subset {\mathbb{R}}^n $, open, bounded,
\begin{align*}
\frac{d}{d \varepsilon}|_{ \varepsilon =0} J_{\Omega} (u + \varepsilon \phi) = 0 \;\;\;,\; \forall \: \phi \in C_0^1 (\Omega)
\end{align*}
where
\begin{align}
W(u) =
\begin{cases}
W^{ \overline{\alpha}} (u) := \prod_{k=1}^{m+1} |u- a_k|^{{\alpha}_k} \;\;\;,\; \overline{\alpha} = ({\alpha}_1,...,{\alpha}_{m+1}) \;\:, \; 0 < {\alpha}_k \leq 2 \\  W^0 (u) := {\chi}_{\lbrace u \in S_A \rbrace}
\end{cases}
\end{align}
and $ S_A $ defined as the interior of the simplex with vertices $ a_1,...,a_m,a_{m+1} $,
\begin{align}
S_A := \lbrace \sum_{i=1}^{m+1} {\lambda}_i a_i \;\: ; \; {\lambda}_i \in [0,1) \:, \; \forall i=1,...,m+1 \:,\; \sum_{i=1}^{m+1} {\lambda}_i =1 \rbrace
\end{align}
Then
\begin{align}
u(x) \in {\overline{S}}_A \;,\; x \in {\mathbb{R}}^n 
\end{align}

For $ \alpha_k \in [0,1) $ we require that $ u $ in addition is a minimizer in the sense of \eqref{ELequation}, so that \eqref{eq:UnHolderBound} is available. $ \\ $
\end{proposition}
\begin{proof}

Following an idea from \cite{CGS} we introduce the set

1. $ {\alpha}_k \in (0,1) \;,\: k=1,..,m. $
\begin{align}
F_M := \lbrace u: {\mathbb{R}}^n \rightarrow {\mathbb{R}}^m \;,\; u \;\: \textrm{minimizer of} \;\: J \;,\; |u(x)| \leq M \rbrace
\end{align}
By Lemma \ref{LemHC} we have the uniform H\"older estimate
\begin{align}\label{eq:UnHolderBound}
| u |_{C^{\beta}( {\mathbb{R}}^n ; {\mathbb{R}}^m)} \leq C(M) \;\;\;\;, u \in F_M
\end{align}
Let $ \Pi $ be the face of the simplex $ {\overline{S}}_A $ defined by $ a_2,...,a_{m+1} $, oppposite to $ a_1 $ and let $ e \perp \Pi $.

Set
\begin{align}
P(u;x) = \langle u(x) - a_2 , e \rangle \;\;\;\;\;\;\;\;\;\;\;\;\;\;\;\;\;\;\;\;\;\;\;\;\;\;\;\;\;\;\;\;\;\;\;\;\;\;\;\;\;\;\;
\end{align}
where $ \langle \cdot , \cdot \rangle $ is the inner product in $ {\mathbb{R}}^m $ and the orientation of $ e $ is such that $ \langle a_2-a_1 ,e \rangle >0 $. Set
\begin{align*}
P_M := \sup \lbrace P(u;x) \; : \; u( \cdot ) \in F_M \;,\; x \in {\mathbb{R}}^n \rbrace
\end{align*}

\textbf{Claim:} $ \; P_M \leq 0 \\ $
Clearly the proposition follows from this claim. We proceed by contradiction. Suppose $ P_M >0 $. Thus there is $ \lbrace u_k \rbrace \in F_M \;,\; \lbrace x_k \rbrace \subset {\mathbb{R}}^n $, such that
\begin{align}
P_M - \frac{1}{k} \leq P(u_k, x_k) \leq P_M.
\end{align}
Set
\begin{align}
v_k (x):= u_k (x+ x_k),
\end{align}
and note that $ v_k \in F_M $ and
\begin{align}
P_M - \frac{1}{k} \leq P(v_k, 0) \leq P_M
\end{align}
By \eqref{eq:UnHolderBound},
\begin{align*}
| v_k |_{C^{\beta}( {\mathbb{R}}^n ; {\mathbb{R}}^m)} \leq C(M)
\end{align*}
hence by Arzela- Ascoli for a subsequence
\begin{align}
v_k \xrightarrow{{C^{\beta}}} v \;\;\;,\; \textrm{on compacts}
\end{align}
We have
\begin{align}
P(v; x) \leq P_M = P(v; 0) > 0 \;\:,\: x \in {\mathbb{R}}^n
\end{align}
By the continuity of $ v $ there is $ R >0 $ such that
\begin{align}
\frac{P_M}{2} \leq P(v; x) \leq P_M \;\:,\: x \in B(0;R)
\end{align}
\begin{align}
P(v_k ;x) = \langle v_k(x) -a_2, e \rangle \geq \frac{P_M}{4} \;\:,\: \textrm{on} \;\: B(0;R)
\end{align}
for $ k $ large.

Thus $ v_k (x) $ uniformly away from $ a_1,...,a_m,a_{m+1} $, we have
\begin{align}
\Delta v_k - W_u (v_k) = 0 \;\:,\: \textrm{in} \;\: B(0;R)
\end{align} 
classically, since $ W_u(u) \in C^1 $ away from $ a_1,...,a_m,a_{m+1} $ and $ x \mapsto W_u(v_k(x)) $ Holder by (A.10), thus $ u \in C^{2+ \beta} (B(0;R)) $.

We now calculate:
\begin{align*}
\Delta P = \langle \Delta v, e \rangle = \langle W_u(u),e \rangle
\end{align*}
\begin{align*}
\frac{\partial}{\partial v_j}W(v) = \frac{\partial}{\partial v_j}({\prod}_{ \nu =1}^{m+1} | v-a_{\nu} |^{{\alpha}_{\nu}} ) = \sum_{i=1}^{m+1} \frac{\partial}{\partial v_j} (| v-a_i|^{{\alpha}_i}) {\prod}_{ \nu \neq i} | v-a_{\nu}|^{ {\alpha}_{\nu}}
\end{align*}
Notice that
\begin{align*}
\frac{\partial}{\partial v_j}(| v-a_i |^2)^{\frac{{\alpha}_i}{2}} = {\alpha}_i | v-a_i |^{{\alpha}_i -2} \cdot (v_j - a_i^j)
\end{align*}
where $ a_i= (a_i^1,...,a_i^m) $

Hence
\begin{align*}
\begin{gathered}
W_v(v) = {\nabla}_v W(v) = \sum_{i=1}^{m+1} a_i (| v - a_i |^{ {\alpha}_i -2}) (v-a_i) \prod_{ \nu \neq i} | v -a_{\nu} |^{{\alpha}_{\nu}} = \\ = {\alpha}_2 | v- a_2 |^{{\alpha}_2 -2} (v-a_2) \prod_{\nu \neq 2} | v-a_{\nu} |^{{\alpha}_{\nu}} + \sum_{i \neq 2} {\alpha}_i | v-a_i |^{{\alpha}_i -2} (v-a_i) \prod_{\nu \neq i} | v- a_{\nu} |^{{\alpha}_{\nu}}.
\end{gathered}
\end{align*}

Therefore
\begin{align*}
\begin{gathered}
\Delta P = {\alpha}_2 | v- a_2 |^{{\alpha}_2 -2} \prod_{\nu \neq 2} | v-a_{\nu} |^{{\alpha}_{\nu}} \langle v-a_2 , e \rangle \\ + \sum_{i \neq 2} {\alpha}_i | v-a_i |^{{\alpha}_i -2} \langle v-a_i, e\rangle \prod_{\nu \neq i} | v- a_{\nu} |^{{\alpha}_{\nu}}
\end{gathered}
\end{align*}
$ \\ $
Note that by the contradiction hypothesis, $ \langle v(x) -a_i ,e \rangle >0 $ (think of $ a_2 $ as the origin).

Hence $ \Delta P >0 $ on $ B(0;R) $ contradicting that $ P(v;x) $ takes its maximum at $ x=0 . \\ $

2. $ \overline{\alpha} = 0 \\ $ For $ W(u) = W^0 (u) := {\chi}_{\lbrace u \in S_A \rbrace} $, the proof proceeds similarly. The difference here is that $ \Delta P = 0 $, in $ B(0;R) $ which also leads to a contradiction by the maximum principle since $ P(v;x) $ takes its maximum at $ x=0 . $

3. $ \alpha_k \in [1,2],\; k=1,...,m. \\ $
In this case we define
\begin{align*}
F_M := \lbrace u : {\mathbb{R}}^n \rightarrow {\mathbb{R}}^m \: ,\; \Delta u - W_u(u) =0 \;,\: | u(x) | \leq M \rbrace
\end{align*}
$ u $ a weak $ W^{1,2} $ solution. By linear elliptic theory we have the estimate \eqref{eq:UnHolderBound}. The rest of the argument is as before. $ \\ $
The proof of the proposition is complete.

\end{proof}

$ \\ \\ \\ $

\section{The free boundary}
\label{sec:appendixB}

We follow closely the formal derivation from \cite{AFS} p.140. We imbed the minimizer in a  class of variations, $u(\tau):=u(\cdot,\tau)$,  with $u(0)$ corresponding to the minimizer, $u(\tau)=u(0)$ outside a ball $B$ centered at some $x_0$ and quite arbitrary otherwise.

Let 

\begin{equation}
U(\tau):=\{ |u(\cdot,\tau)-a|>0\}
\end{equation} for 

$$
a\in \{W=0\},\, u(\tau)=a\textrm{ on }\partial U(\tau)
$$

$ \\ $

Set 

\begin{equation}
\lambda(\tau):=\frac{1}{2}\int_{U(\tau)}|\nabla u(\tau)|^2\,dx \: , \; \mu(\tau):=\int_{U(\tau)} W(u(\tau))\,dx
\end{equation}

We denote $V:=\frac{\partial X}{\partial \tau}$ where $X(s,\tau)$ is a parametrisation of $\partial U(\tau), \,s\in\Omega\subset \mathbb{R}^{n-1}$.

Then we have:

\begin{align}\label{eq:taudot}
\dot\lambda(\tau)=&\int_{U(\tau)} \nabla u(\tau)\nabla u_\tau(\tau)\,dx+\dfrac{1}{2}\int_{\partial U(\tau)}|\nabla u(\tau)|^2 V\cdot \nu dS\non\\
=&\int_{U(\tau)} -\Delta u(\tau) u_\tau (\tau)\,dx+\int_{\partial U(\tau)}\frac{\partial u}{\partial \nu}\cdot u_\tau\,dS+\frac{1}{2} \int_{\partial U(\tau)} |\nabla u(\tau)|^2 V\cdot\nu\,dS
\end{align} 
where $ \nu $ is the unit outward normal to $ \partial U( \tau) $ (pointing outside $ U(\tau) $).

Now from $ u(X(s,\tau),\tau)=a$ we obtain:

\begin{align}
0=& \frac{\partial }{\partial\tau}[u(X(s,\tau),\tau)]=\frac{\partial u}{\partial \tau}+\frac{\partial u}{\partial\nu}\frac{\partial X}{\partial \tau}\cdot \nu\non\\
=&u_\tau+\frac{\partial u}{\partial \nu} V\cdot\nu 
\end{align}

Hence
\begin{equation}\label{eq:utauunu}
u_\tau\cdot\frac{\partial u}{\partial \nu}=-|\frac{\partial u}{\partial\nu}|^2 V\cdot\nu
\end{equation}

Then from \eqref{eq:taudot} and \eqref{eq:utauunu} and the equation $ \Delta u = W_u(u) $ we get:

\begin{equation}
\dot\lambda(0)=\int_{U(0)} -W_u(u(0))u_\tau(0)\,dx-\frac{1}{2}\int_{\partial U(0)}|\nabla u(0)|^2 V\cdot\nu dS.
\end{equation}

On the other hand

\begin{equation}
\dot \mu(\tau)=\int_{\partial U(\tau)} W(u(\tau)) V\cdot\nu dS+\int_{U(\tau)} W_u(u(\tau))u_{\tau}(\tau) \,dx
\end{equation}

Here for $0<\alpha<2$ utilizing that $W(u(0))=0$ on $\partial U(0)$ we get:

\begin{align}
0=&\dot \mu(0)+\dot \lambda (0)\non\\
  =&-\frac{1}{2} \int_{\partial U(0)} |\nabla u(0)|^2 V\cdot\nu\,dS
\end{align} and since $V$ is arbitrary

\begin{equation}
|\nabla u(0)|=0\,\textrm{ on }\partial U(0)\textrm{ for }\alpha\in (0,2).
\end{equation}
(we note that $u\in C^{1,\beta-1}$,$\beta=\frac{2}{2-\alpha}$ by \cite{AGZ}).

Now, for $\alpha=0$ we have $W(u(0))=1$  on $\partial U(0)$ and 
\begin{align}
0=&\dot \mu(0)+\dot \lambda (0)\non\\
  =& \int_{\partial U(0)} V\cdot\nu\,dS-\frac{1}{2} \int_{\partial U(0)} |\nabla u(0)|^2 V\cdot\nu\,dS
\end{align}
hence $\frac{1}{2}|\nabla_+ u(0)|^2=1$ ($u$ is only Lipschitz, $\nabla_+$ is the one-sided gradient).
\end{appendix}

\end{document}